\documentclass[11pt]{amsart}

\usepackage{xypic,amsmath,amssymb}

\newtheorem{proposition}{Proposition}[section]
\newtheorem{lemma}[proposition]{Lemma}
\newtheorem{corollary}[proposition]{Corollary}
\newtheorem{theorem}[proposition]{Theorem}

\theoremstyle{definition}
\newtheorem{definition}[proposition]{Definition}

\theoremstyle{remark}
\newtheorem{remark}[proposition]{Remark}

\newcommand{\thlabel}[1]{\label{th:#1}}

\newcommand{\selabel}[1]{\label{se:#1}}
\newcommand{\seref}[1]{Section~\ref{se:#1}}
\newcommand{\lelabel}[1]{\label{le:#1}}
\newcommand{\leref}[1]{Lemma~\ref{le:#1}}
\newcommand{\prlabel}[1]{\label{pr:#1}}
\newcommand{\prref}[1]{Proposition~\ref{pr:#1}}
\newcommand{\colabel}[1]{\label{co:#1}}
\newcommand{\coref}[1]{Corollary~\ref{co:#1}}
\newcommand{\relabel}[1]{\label{re:#1}}
\newcommand{\reref}[1]{Remark~\ref{re:#1}}

\newcommand{\delabel}[1]{\label{de:#1}}

\newcommand{\eqlabel}[1]{\label{eq:#1}}
\newcommand{\equref}[1]{(\ref{eq:#1})}

\def\equal#1{\smash{\mathop{=}\limits^{#1}}}

\newcommand{\Hom}{{\rm Hom}}

\newcommand{\Aut}{{\rm Aut}\,}

\def\ot{\otimes}

\def\NN{{\mathbb N}}

\def\ZZ{{\mathbb Z}}

\newcommand{\Cc}{\mathcal{C}}
\newcommand{\Dd}{\mathcal{D}}

\newcommand{\Hh}{\mathcal{H}}
\newcommand{\Mm}{\mathcal{M}}

\def\*C{{}^*\hspace*{-1pt}{\Cc}}

\def\text#1{{\rm {\rm #1}}}

\def\ol{\overline}

\allowdisplaybreaks[4]

\begin{document}
\title[Monoidal Hom-Hopf Algebras]{Monoidal Hom-Hopf Algebras}

\author{S. Caenepeel}
\address{Faculty of Engineering,
Vrije Universiteit Brussel, Pleinlaan 2, B-1050 Brussels, Belgium}
\email{scaenepe@vub.ac.be}
\urladdr{http://homepages.vub.ac.be/\~{}scaenepe/}
\author{I. Goyvaerts}
\address{Faculty of Engineering,
Vrije Universiteit Brussel, Pleinlaan 2, B-1050 Brussels, Belgium}
\email{Isar.Goyvaerts@vub.ac.be}



\begin{abstract}
Hom-structures (Lie algebras, algebras, coalgebras, Hopf algebras) have been
investigated in literature recently. We study Hom-structures from the point of
view of monoidal categories; in particular, we introduce a symmetric monoidal
category such that (certain classes of) Hom-algebras coincide with algebras in this monoidal category.
Similar properties for Hom-coalgebras, Hopf and Lie algebras hold.
\end{abstract}
\maketitle

\section*{Introduction}\selabel{0}
Hom-Lie algebras were introduced in \cite{HS}, motivated by the study of deformations of the Witt and the Virasoro algebras. They have been investigated
by various authors ever since. The idea is that the Jacobi identity is replaced by
the so-called Hom-Jacobi identity, namely
\begin{equation}
[\alpha(x),[y,z]]+[\alpha(y),[z,x]]+ [\alpha(z),[x,y]]=0,
\end{equation}
where $\alpha$ is a linear endomorphism of the underlying vector space. Hom-algebras have been
introduced in \cite{MS}. Now the associativity is replaced by Hom-associativity:
\begin{equation}
\alpha(a)(bc)=(ab)\alpha(c).
\end{equation}
Hom-coassociativity for a Hom-coalgebra can be considered in a similar way,
see \cite{MS2,MS3}. Definitions of Hom-bialgebras and Hom-Hopf algebras also
have been proposed, see \cite{MS2,MS3,Yau3}. 
\\The issue is that a variety of
different definitions is possible, and that it is not always evident to decide from the existing literature what are the ``right" definitions to use.
For example, sometimes it is needed that $\alpha$ is multiplicative (that is,
$\alpha[a,b]=[\alpha(a),\alpha(b)]$ in the Lie case, and $\alpha(ab)=\alpha(a)
\alpha(b)$ in the algebra case). In the definition of a Hom-bialgebra, it is not
clear whether one should take different endomorphisms describing the Hom-associativity
and the Hom-coassociativity; see \cite{MS3}, where different endomorphisms
are allowed, and \cite{Yau3}, where the two endomorphisms are the same.
In \cite{Fregier}, the authors compare all possible variations of the definition of
a Hom-algebra.\\
The aim of this paper is to understand Hom-structures from the point of view of
monoidal categories. This leads in a natural way to the definition of \textit{monoidal} Hom-algebras,
Hom-coalgebras, etc. These monoidal Hom-structures are certain classes of the Hom-structures found in literature. It turns out that the monoidal categorical point of view can help us, amongst other things, to point out what definitions are useful for certain types of applications.
We will construct a symmetric monoidal category, and then
introduce monoidal Hom-algebras, Hom-coalgebras etc. as algebras, coalgebras etc.
in this monoidal category. The remarkable thing is that we have to consider a
category in which the associativity constraint is non-trivial.\\
In \seref{1},  we introduce the category $\Hh(\Cc)$
associated to a monoidal category $\Cc$; the associativity and unit constraints come
in two versions, leading to two different monoidal structures on $\Hh(\Cc)$, denoted
$\Hh(\Cc)$ and $\widetilde{\Hh}(\Cc)$. The objects of the latter consist of pairs formed by an object of $\Cc$
together with an automorphism of this object. The associativity constraint is given
by \equref{1.1.3}, and involves the automorphisms and their inverses. Actually, $\Hh(\Cc)$ and $\widetilde{\Hh}(\Cc)$
are tensor isomorphic (\prref{1.7}), leading to structure theorems for algebras,
coalgebras, ... in $\widetilde{\Hh}(\Cc)$.
\\In \seref{2}, we apply
this construction to the category of $k$-modules, and look at algebras in the category
$\widetilde{\Hh}(\Mm_k)$. For the associativity,
we obtain the formula
\begin{equation}
(ab)c=\alpha(a)(b\alpha^{-1}(c)),
\end{equation}
which is clearly equivalent to the aforementioned condition for Hom-associativity. By definition, the
automorphism $\alpha$ of the algebra $A$ has to be invertible, and it has to
be multiplicative, because the multiplication map of $A$ has to be a morphism
in the category $\widetilde{\Hh}(\Mm_k)$. Although the definitions in \cite{MS,MS2,MS3} do not require $\alpha$ to be invertible, neither multiplicative, it seems that, in order to have certain ``nice" properties for Hom-algebras, Hom-coalgebras, ...,
we need these extra conditions. This is illustrated by some of our results; for
example, we prove in \prref{2.4.4} that the category of modules over a monoidal Hom-bialgebra
is monoidal, and the proof does not work if $\alpha$ is not multiplicative. In \seref{2a}, we generalize the Fundamental Theorem for
Hopf-modules to the Hom-setting. Another example of an application of the monoidal categorical point of view in the Hom-world is presented in \seref{3}, where we consider a Hom-version of the classical group algebra (applying our construction to the category of
sets). In this case, the monoidal categorical framework offers a possibility to see what the unitality conditions look like. In \seref{4}, as a special case of Lie algebras in the
additive symmetric monoidal category $\widetilde{\Hh}(\Mm_k)$, we recover monoidal Hom-Lie algebras.
Compared to the Hom-Lie algebras considered in literature, they also satisfy
the Hom-Jacobi identity; the additional requirement is that $\alpha$ is multiplicative and invertible. 
As another example of a Hom-Hopf algebra, we introduce the tensor Hom-algebra in
\seref{5}. This is applied in \seref{7} to the construction of the universal enveloping
(Hom-Hopf) algebra of a (monoidal) Hom-Lie algebra.

\section{Preliminary results}\selabel{Hopfalgbraidedmonoidal}
We assume some (at least working) knowledge of (braided) monoidal categories and categorical constructions. In this respect, \cite{K} and \cite{ML} can be useful references.
\\Let us start by mentioning some results concerning the behaviour of some algebraic objects in (braided) monoidal categories under (braided) monoidal functors. This will be used later on. 

\begin{definition}
Let $\Cc=(\Cc,\ot,I,a,l,r)$ be a monoidal category. An {\em algebra}\index{algebra} (in literature sometimes also called {\em monoid}) in $\Cc$ is a triple
$A=(A,m,\eta)$, where $A\in \Cc$ and $m: A\ot A\to A$ and $\eta: I\to A$ are
morphisms in $\Cc$ such that the following diagrams commute:\index{unitality}\index{associativity}
$$
\xymatrix{
(A\ot A)\ot A\ar[rr]^{m\ot A}\ar[d]_{(A\ot m)\circ a_{A,A,A}}&& A\ot A\ar[d]^{m}\\
A\ot A\ar[rr]^{m}&&A}~~~~~~
\xymatrix{I\ot A\ar[r]^{u\ot A}\ar[rd]_{l_A}&
A\ot A\ar[d]^{m}&A\ot I\ar[l]_{A\ot u}\ar[dl]^{r_A}\\
&A&}
$$
\end{definition}
We can define a {\em coalgebra} (or {\em comonoid}) in a monoidal category $\Cc$\index{coalgebra} as an algebra in the opposite category $\Cc^{\rm op}$. Here, $\Cc^{op}=(\Cc^{op},\ot^{op},I)$ denotes the opposite category of $\Cc$ and $\ot^{op}: \Cc^{op}\times\Cc^{op}\to \Cc^{op}$ the opposite tensor product functor induced in the obvious way by $\ot$.

Let $\Cc=(\Cc,\ot,I,c)$ be a braided monoidal category (in order to make our notation not too heavy, we will consider $\Cc$ to be strict). 
A {\em bialgebra} (or {\it bimonoid})\index{bimonoid} in $\Cc$ is a pentuple
$H=(H,m,\eta,\Delta,\varepsilon)$, where $(H,m,\eta)$ is an algebra and 
$(H,\Delta,\varepsilon)$ a coalgebra in $\Cc$, such that $\Delta:H\to H\ot H$ and $\varepsilon: H\to I$ are algebra morphisms, or, equivalently, such that $m:H\ot H\to H$ and $\eta:I\to H$ are coalgebra morphisms. Here the algebra structure on $H\ot H$ is given by 
\begin{equation}
\xymatrix{H\ot H\ot H\ot H\ar[rr]^-{H\ot c_{H,H}\ot H} && H\ot H\ot H\ot H\ar[r]^-{m\ot m} & H\ot H},
\end{equation}
and the coalgebra structure by 
\begin{equation}
\xymatrix{H\ot H\ar[r]^-{\Delta\ot \Delta} & H\ot H\ot H\ot H\ar[rr]^-{H\ot c_{H,H}\ot H}&& H\ot H\ot H\ot H}.
\end{equation}
Let $C=(C,\Delta_C,\varepsilon_C)$ be a coalgebra, and $A=(A,m_A,\eta_A)$ an algebra in $\Cc$. We can now define
a product $*$ on $\Hom_\Cc(C,A)$ as follows: for $f,g: C\to A$, let
$f*g$ be defined by
$f*g = m_A\circ(f\ot g)\circ\Delta_C$.
$*$ is called the {\it convolution product}\index{product!convolution}.
Observe that $*$ is associative, and that for any 
$f\in\Hom_\Cc(C,A)$, we have that
$f*(\eta_A\circ\varepsilon_C)=(\eta_A\circ\varepsilon_C)*f=f,$
so the convolution product makes $\Hom_\Cc(C,A)$ into a monoid ($\eta_A\circ\varepsilon_C$ being the unit element for $*$).

Now suppose that $H$ is a bialgebra in $\Cc$, and take $H=A=C$ in the above construction.
If the identity $H$ of $H$ has a convolution inverse $S=S_H$, then we
say that $H$ is a {\em Hopf algebra}\index{Hopf algebra} (or {\em Hopf monoid})\index{Hopf monoid} in $\Cc$. $S$ is called the 
{\em antipode}\index{antipode} of $H$.

\begin{proposition}\prlabel{1.8}
Let $(F,\varphi_0,\varphi_2)$ be a monoidal functor between the monoidal
categories $\Cc=(\Cc,\ot,I,a,l,r)$ and $\Dd=(\Dd,\odot, J,a',l',r')$. If $A$ is an
algebra in $\Cc$, then $F(A)$ is an algebra in $\Dd$, with
$$m_{F(A)}=F(m_A)\circ \varphi_2(A,A)~~;~~\eta_{F(A)}=F(\eta)\circ \varphi_0.$$
In a similar way, if $(F,\psi_0,\psi_2)$ is a comonoidal functor from $\Cc$ to $\Dd$,
and $C$ is a coalgebra in $\Cc$, then $F(C)$ is a coalgebra in $\Dd$ with
$$\Delta_{F(C)}=\psi_2(C,C)\circ F(\Delta_C)~~;~~\varepsilon_{F(C)}=\psi_0
\circ F(\varepsilon_C).$$
\end{proposition}

Assume that $F$ is strong monoidal and braided, and that $H$ is a bialgebra in $\Cc$. It follows
from \prref{1.8} that $F(H)$ is an algebra and a coalgebra in $\Dd$, and it is easy to
show that it is even a bialgebra in $\Dd$. Let us now show that we have a similar
property for Hopf algebras.

\begin{lemma}\lelabel{1.11}
Let $F:\ \Cc\to \Dd$ be a strong monoidal functor, $A$ an algebra in $\Cc$
and $C$ a coalgebra in $\Cc$. Then for $f,g\in \Hom_\Cc(C,A)$, we have that
$$F(f*g)=F(f)*F(g).$$
\end{lemma}

\begin{proof}
From the naturality of $\varphi_2$, it follows that the diagram
\begin{equation}\eqlabel{1.11.1}
\xymatrix{
F(C)\odot F(C)\ar[rr]^{\varphi_2(C,C)}\ar[d]_{F(f)\odot F(g)}&&
F(C\ot C)\ar[d]^{F(f\ot g)}\\
F(A)\odot F(A)\ar[rr]^{\varphi_2(A,A)}&&F(A\ot A)}
\end{equation}
commutes. Then we easily compute that
\begin{eqnarray*}
&&\hspace*{-2cm}
F(f)*F(g)=m_{F(A)}\circ (F(f)\odot F(g))\circ \Delta_{F(C)}\\
&=& F(m_A)\circ \varphi_2(A,A)\circ (F(f)\odot F(g))\circ \varphi_2^{-1}(C,C)\circ F(\Delta_C)\\
&\equal{\equref{1.11.1}}&
F(m_A\circ (f\ot g)\circ \Delta_C)=F(f*g).
\end{eqnarray*}
\end{proof}

As a consequence, we immediately obtain the following result.

\begin{proposition}\prlabel{1.12}
Let $(F,\varphi_0,\varphi_2)$ be a strong monoidal braided functor between the
braided monoidal categories $\Cc$ and $\Dd$. If $H$ is a Hopf algebra in $\Cc$,
then $F(H)$ is a Hopf algebra in $\Dd$.
\end{proposition}

\begin{proof}
We have already seen that our result holds for bialgebras. If $S$ is an antipode for $H$,
then it follows from \leref{1.11} that $F(S)$ is an antipode for $F(H)$.
\end{proof}

Symmetric monoidal categories are well-suited to consider cyclic permutations on copies of the same object. Moreover, if we assume additivity, these categories offer a framework to generalize the classical definition of Lie algebra and incorporate examples of  ``generalized" Lie algebras, see \cite{Khar} for instance. We present another illustration of this fact in \seref{4}.  
\\\\Consider parenthized monomials in $n$ non-commuting variables $X_1,\cdots,X_n$,
such that every variable $X_i$ occurs one time in the monomial. Examples of such
polynomials in de case $n=4$ are
$(X_1X_3)(X_2X_4)$, $X_1(X_4(X_3X_2))$ etc...
Let $U_n$ be the set of all these polynomials. We describe these polynomials.\\
If we delete the parentheses in $P\in U_n$, then we obtain a monomial of the form
$X_{\sigma(1)}X_{\sigma(2)}\cdots X_{\sigma(n)}$, where $\sigma\in S_n$ is a
permutation of $\{1,\cdots,n\}$. Thus we obtain a surjection $U_n\to S_n$.\\
Consider the polynomials $P\in U_n$ such that the corresponding permutation in $S_n$
is the identity; such polynomials are in bijective correspondence to $T_n$, the set
of planar binary trees with $n$ leaves and one root (see for example \cite{Yau1}).\\
Thus we can consider $U_n$ as the direct product $T_n\times S_n$: for $\sigma\in S_n$,
$\psi\in T_n$, take the polyniomial $X_1X_2\cdot X_n$ with parenthesis corresponding
to $\psi$, and then permute the variable $X_1,\cdots,X_n$ using $\sigma$.
The cardinality of $T_n$ is the $(n-1)$-th Catalan number $C_n={(2n-2)!/((n-1)!n!)}$,
$\#(S_n)=n!$, so we find that $\#(U_n)={(2n-2)!/(n-1)!}$. $U_n$ is a right $S_n$-set:
 $(\psi,\sigma)\sigma'=(\psi,\sigma\circ \sigma')$. We embed $T_n$ in $U_n$ by
 identifying $\psi$ and $(\psi,e)$.\\
Now assume that $\Cc$ is a symmetric monoidal category with symmetry $c$, and take
$u=\psi\sigma\in U_n$. 
We have functors
$$p_\sigma:\ \Cc^n\to \Cc^n,~~p_\sigma(M_1,\cdots,M_n)=(M_{\sigma(1)},\cdots,
M_{\sigma(n)})$$
and
$\psi(\ot):\ \Cc^n\to \Cc$, where $\psi(\ot)(M_1,\cdots,M_n)$ is the $\psi$-parenthized
$n$-fold tensor product $M_1\ot M_2\ot\cdots\ot M_n$. Let $\ot^u=
\psi(\ot)\circ p_\sigma:\ \Cc^n\to \Cc$. Take $u,u'\in U_n$. Using compositions of the associativity
constraint $a$ and the symmetry $c$, tensored up with some identity natural transformations, 
we can construct a natural transformation
$$b(u,u'):\ \ot^u\to \ot^{u'}.$$
It follows from the coherence conditions of $a$ and $c$ that this natural transformation is
unique.\\
We fix an element $t^{n}\in T_n$ corresponding to the following parenthized monomial:
$$X_1(X_2(\cdots (X_{n-1}X_n)\cdots)).$$
$t^n$ can be defined recursively as follows: $t^n=t^1 \vee t^{n-1}$, where $t^1$ is
the unique element of $T^1$, and $\vee$ denotes the grafting of trees, see
\cite{Yau1} e.g..\\
Let $\tau$ be the nontrivial element of $S_2$, and $s\in S_3$ the cyclic permutation
$s(1)=2$, $s(2)=3$, $s(3)=1$. Then we easily compute
$$b(t^2, t^2\tau)=c:\ \ot^{t^2}=\ot\to \ot^{t^2 \tau }=\ot\circ p_\tau;$$
$$b(t^3, t^3 s)=a\circ (c\ot 1_\Cc)\circ a^{-1}\circ (1_\Cc\ot c):\
\ot^{t^3}\to \ot^{t^3 s};$$
$$b(t^3, t^3 s^2)=(1_\Cc\ot c)\circ a\circ (c\ot 1_\Cc)\circ a^{-1}:\
\ot^{t^3}\to \ot^{t^3 s^2}.$$
In particular, for $L\in \Cc$, we have the following morphisms in $\Cc$:
$$b(t^3, t^3 s)_{L,L,L},~b(t^3, t^3 s^2)_{L,L,L}:\ L\ot (L\ot L)\to L\ot (L\ot L).$$ 
Let us denote $t_c:=b(t^3, t^3 s)_{L,L,L}$ and $w_c:=b(t^3, t^3 s^2)$, just to lighten notation in the following
\begin{definition}\delabel{LieAlgsymm}
A \textit{Lie algebra in} $\Cc$\index{Lie algebra} is a couple $(L,[-,-])$, denoted $L$ for short if there is no confusion possible, where $L$ is an object of $\Cc$ and $[-,-]:\ L\ot L\to L$ is a morphism (which we call a {\em Lie bracket})\index{Lie bracket} in $\Cc$ that satisfies\index{antisymmetry}\index{identity!Jacobi}
\begin{eqnarray}\eqlabel{ASsymm}
[-,-] \circ (1_{L\ot L} + c_{L,L})&=&0_{L\ot L,L},\\
\eqlabel{Jacsymm}
[-,-]\circ (1_{L}\ot [-,-])\circ (1_{L\ot(L\ot L)}+ t_c+ w_c)&=&0_{L\ot (L\ot L),L}.
\end{eqnarray}
\end{definition}
\section{The Hom-construction}\selabel{1}
Let $\Cc$ be a category. We introduce a new category $\Hh(\Cc)$ as follows:
objects are ordered pairs $(M,\mu)$, with $M\in \Cc$ and $\mu\in \Aut_\Cc(M)$.
A morphism $f:\ (M,\mu)\to (N,\nu)$ is a morphism $f:\ M\to N$ in $\Cc$ such that
\begin{equation}\eqlabel{1.1.1}
\nu\circ f=f\circ \mu.
\end{equation}
$\Hh(\Cc)$ will be called the \textit{Hom-category}\index{category!Hom-} associated to $\Cc$. If $(M,\mu)\in \Cc$,
then $\mu:\ M\to M$ is obviously a morphism in $\Hh(\Cc)$.\\
Now assume that $\Cc=(\Cc,\ot,I,a,l,r)$ is a monoidal category. It is easy to show that
$\Hh(\Cc)=(\Hh(\Cc),\ot,(I,I),a,l,r)$ is also a monoidal category. The tensor product of
$(M,\mu)$ and $ (N,\nu)$ in $\Hh(\Cc)$ is given by the formula
\begin{equation}\eqlabel{1.1.2}
(M,\mu)\ot (N,\nu)=(M\ot N,\mu\ot \nu).
\end{equation}
On the level of morphisms, the tensor product is the tensor products of morphisms in $\Cc$.
We will now modify this construction.

\begin{proposition}\prlabel{1.1}
Let $\Cc=(\Cc,\ot,I,a,l,r)$ be a monoidal category. Then
$\widetilde{\Hh}(\Cc)=(\Hh(\Cc),\ot,(I,I),\tilde{a},\tilde{l},\tilde{r})$ is also a monoidal category.
The tensor product is given by \equref{1.1.2}. The associativity constraint $\tilde{a}$ is given 
by the formula
\begin{equation}\eqlabel{1.1.3}
\tilde{a}_{M,N,P}=a_{M,N,P}\circ ((\mu\ot N)\ot \pi^{-1})
=(\mu\ot (N\ot \pi^{-1}))\circ a_{M,N,P},
\end{equation}
for $(M,\mu),(N,\nu),(P,\pi)\in \Hh(\Cc)$. The unit constraints $\tilde{l}$ and
$\tilde{r}$ are given by
\begin{equation}\eqlabel{1.1.4}
\tilde{l}_M=\mu\circ l_M=l_M\circ (I\ot \mu)~~;~~
\tilde{r}_M= \mu\circ r_M=r_M\circ (\mu\ot I).
\end{equation}
\end{proposition}

\begin{proof}
From the naturality of $a$, it follows that ${a}_{M,N,P}$ and $\tilde{a}_{M,N,P}$
are morphisms in $\Hh(\Cc)$. Let us show that $\tilde{a}$ is natural in $M,N$ and $P$.
Let $f:\ M\to M'$, $g:\ N\to N'$ and $h:\ P\to P'$ be morphisms in $\Hh(\Cc)$,
and consider the diagram
$$\xymatrix{
(M\ot N)\ot P\ar[rr]^{(\mu\ot N)\ot \pi^{-1}}\ar[d]^{(f\ot g)\ot h}&&
(M\ot N)\ot P\ar[rr]^{{a}_{M,N,P}}\ar[d]^{(f\ot g)\ot h}&&
M\ot(N\ot P)\ar[d]^{f\ot (g\ot h)}\\
(M'\ot N')\ot P'\ar[rr]^{(\mu'\ot N')\ot {\pi'}^{-1}}&&
(M'\ot N')\ot P'\ar[rr]^{{a}_{M',N',P'}}&&
M'\ot (N'\ot P')}$$
The left square commutes since $f,g,h\in \Hh(\Cc)$. The right square commutes since
$a$ is natural. Hence the whole diagram commutes, and this shows that $\tilde{a}$
is natural.\\
We will next show that $\tilde{a}$ satisfies the Pentagon Axiom (see \cite[(XI.2.6)]{K} e.g.). In the following
computation, the naturality of $a$ is used several times. 
\\For $(M,\mu),(N,\nu),(P,\pi),(Q,q)\in \Hh(\Cc)$, we have
\begin{eqnarray*}
&&\hspace*{-1cm}
\tilde{a}_{M,N,P\ot Q}\circ \tilde{a}_{M\ot N,P,Q}\\
&\equal{\equref{1.1.3}}&
{a}_{M,N,P\ot Q}\circ (\mu \ot N\ot (\pi^{-1}\ot q^{-1}))\circ
{a}_{M\ot N,P,Q}\circ (((\mu\ot \nu)\ot P)\ot q^{-1})\\
&=&
{a}_{M,N,P\ot Q}\circ {a}_{M\ot N,P,Q}\\
&&\hspace*{8mm} \circ
(((\mu\ot N)\ot \pi^{-1})\ot q^{-1})\circ (((\mu\ot \nu)\ot P)\ot q^{-1})\\
&=&
{a}_{M,N,P\ot Q}\circ {a}_{M\ot N,P,Q}\circ
 (((\mu^2\ot \nu)\ot \pi^{-1})\ot q^{-2});\\
 &&\hspace*{-1cm}
(M\ot \tilde{a}_{N,P,Q})\circ \tilde{a}_{M,N\ot P,Q}
\circ (\tilde{a}_{M,N,P}\ot Q)\\
 &\equal{\equref{1.1.3}}&
 (M\ot {a}_{N,P,Q})\circ (M\ot ((\nu\ot P)\ot q^{-1}))\circ {a}_{M,N\ot P,Q}\\
&&\hspace*{8mm}
 \circ ((\mu\ot (N\ot P))\ot q^{-1})
\circ ({a}_{M,N,P}\ot Q)\circ (((\mu\ot N)\ot \pi^{-1})\ot Q)\\
  &=&(M\ot \tilde{a}_{N,P,Q})\circ {a}_{M,N\ot P,Q}
   \circ (a_{M,N,P}\ot Q)\circ (((\mu^2\ot \nu)\ot \pi^{-1})\ot q^{-2})
  \end{eqnarray*}
  Both expressions are equal, since the Pentagon Axiom holds for $a$.\\
  We will now show that $\tilde{l}$ is natural: take $f:\ M\to N$ in $\Hh(\Cc)$,
  and consider the diagram
  $$\xymatrix{
  I\ot M \ar[rr]^{l_M}\ar[d]^{I\ot f}&&M\ar[rr]^\mu\ar[d]^{f}&& M\ar[d]^{f}\\
  I\ot N \ar[rr]^{l_N}&& N\ar[rr]^{\nu}&&N}$$
  The left square commutes, by the naturality of $l$, and the right square commutes
  since $f\in \Hh(\Cc)$. The naturality of $\tilde{r}$ can be proved in a similar way.\\
  Let us finally show that the Triangle Axiom is satisfied.
  \begin{eqnarray*}
  &&\hspace*{-2cm}
  (M\ot \tilde{l}_N)\circ \tilde{a}_{M,I,N}
  \equal{(\ref{eq:1.1.3},\ref{eq:1.1.4})}
  (M\ot \nu)\circ  (M\ot {l}_N)\circ (\mu\ot (I\ot \nu^{-1}))\circ {a}_{M,I,N}\\
  &=&
  (M\ot \nu)\circ (\mu\ot \nu^{-1})\circ (M\ot l_N)\circ a_{M,I,N}\\
  &=& (\mu\ot N)\circ (r_M\ot N)= \tilde{r}_M\ot N.
  \end{eqnarray*}
\end{proof}

\begin{proposition}\prlabel{1.2}
Let $\Cc=(\Cc,\ot,I,a,l,r,c)$ be a braided monoidal category. Then
$\widetilde{\Hh}(\Cc)=(\Hh(\Cc),\ot,(I,I),\tilde{a},\tilde{l},\tilde{r},c)$ is also a 
braided monoidal category.
\end{proposition}

\begin{proof}
It follows from the naturality of the braiding $c$ that $c_{M,N}$ is a morphism
in $\Hh(\Cc)$, for all $(M,\mu),(N,\nu)\in \Hh(\Cc)$. $c$ is still natural if we view
it as a natural isomorphism $\ot \to \ot\circ \tau$ between functors $\Hh(\Cc)\times
\Hh(\Cc)\to \Hh(\Cc)$. Now let us verify the Hexagon Axioms (H1,H2) from
\cite[XIII.1.1]{K}. We restrict attention to (H1), the proof of (H2) is similar. We need
to show that the following diagram commutes, for any
$(M,\mu),(N,\nu),(P,\pi)\in \Hh(\Cc)$
$$\xymatrix{
(M\ot N)\ot P\ar[rr]^{\tilde{a}_{M,N,P}}\ar[d]^{c_{M,N\ot P}}&&
M\ot (N\ot P)\ar[rr]^{c_{M,N}\ot P}&&
(N\ot P)\ot M\ar[d]^{\tilde{a}_{N,P,M}}\\
(N\ot M)\ot P\ar[rr]^{\tilde{a}_{N,M,P}}&&
N\ot (M\ot P)\ar[rr]^{N\ot c_{M,P}}&&
N\ot (P\ot M)}$$
Indeed,
\begin{eqnarray*}
  &&\hspace*{-1cm}
  \tilde{a}_{N,P,M}\circ c_{M,N\ot P} \circ \tilde{a}_{M,N,P}\\
  &=&
  {a}_{N,P,M}\circ ((\nu\ot P)\ot \mu^{-1})\circ c_{M,N\ot P} \circ {a}_{M,N,P}
  \circ ((\mu \ot N)\ot \pi^{-1})\\
  &=&
   {a}_{N,P,M}\circ c_{M,N\ot P}\circ (\mu^{-1}\ot (\nu\ot P)) \circ {a}_{M,N,P}
  \circ ((\mu \ot N)\ot \pi^{-1})\\
  &=&
   {a}_{N,P,M}\circ c_{M,N\ot P}\circ {a}_{M,N,P}\circ ((M\ot \nu)\ot \pi^{-1});\\
   &&\hspace*{-1cm}
   (N\ot c_{M,P})\circ \tilde{a}_{N,M,P} \circ c_{M,N}\ot P\\
   &=&
   (N\ot c_{M,P})\circ {a}_{N,M,P} \circ ((\nu\ot M)\ot \pi^{-1})\circ c_{M,N}\ot P\\
   &=&
   (N\ot c_{M,P})\circ {a}_{N,M,P}\circ c_{M,N}\ot P\circ ((M\ot \nu)\ot \pi^{-1}).
\end{eqnarray*}
The two expressions are equal since $c$ is a braiding on $\Cc$.
\end{proof}
Recall that an object $M$ in a monoidal category $\Cc$ is called {\em left rigid}\index{object!left rigid} if there exists an object $M^*$ together with morphisms $b_{M}:I\to M\ot M^*$ and $d_{M}:M^*\ot M\to I$ such that the following formulas are satisfied: 
\begin{eqnarray}\eqlabel{1.3.1}
&&r_M\circ (M\ot d_M)\circ a_{M,M^*,M}\circ (b_M\ot M)\circ l_M^{-1}=M;\\
&&\eqlabel{1.3.2}
l_{M^*}\circ (d_M\ot M^*)\circ a^{-1}_{M^*,M,M^*}\circ (M^*\ot b_M)\circ r_{M^*}^{-1}=M^*.
\end{eqnarray}
It is easily verified that if $M$ is left rigid, then the object $M^*$ is unique up to isomorphism. In this situation, we call $M^*$ the {\em left dual}\index{object! left dual} of $M$.
\\A {\em right rigid}\index{object!right rigid} object is defined symmetrically. Remark that if $M$ is left rigid with left dual $M^*$, then $M^*$ is right rigid with right dual $M$.\index{object!right dual}
\\Now let us assume that $\Cc$ has left duality. Let $M^*$ be the left dual of $M\in\Cc$,
and $b_M:\ I\to M\ot M^*$, $d_M:\ M^*\ot M\to I$ the coevaluation and evaluation maps.
We will show that $\widetilde{\Hh}(\Cc)$ also has left duality. 
\\We first recall the following properties. For $f\in \Hom_{\Cc}(M,N)$, we have the dual morphism
$$f^*=l_{M^*}\circ (d_N\ot M^*)\circ ((N^*\ot f)\ot M^*)\circ a^{-1}_{N^*,M,M^*}\circ
(N^*\ot b_M)\circ r_{N^*}^{-1}:\ N^*\to M^*,$$
satisfying the properties
\begin{eqnarray}\eqlabel{1.3.3}
d_M\circ (f^*\ot M)&=&d_N\circ (N^*\ot f);\\
\eqlabel{1.3.4}
(f\ot M^*)\circ b_M&=&(N\ot f^*)\circ b_N.
\end{eqnarray}
Let $g:\ N\to P$ be another morphism in $\Cc$. It follows from \equref{1.3.4} that
\begin{equation}\eqlabel{1.3.5}
((g\circ f)\ot M^*)\circ b_M=(g\ot f^*)\circ b_N.
\end{equation}
For $f,g:\ M\to M$, we now have
\begin{eqnarray}
&&\hspace*{-15mm}\nonumber
(d_M\ot M^*)\circ((f^*\ot M)\ot g^*)\circ a^{-1}_{M^*,M,M^*}\circ (M^*\ot b_M)\\
&\equal{\equref{1.3.3}}&
(d_M\ot M^*)\circ((M^*\ot f)\ot g^*)\circ a^{-1}_{M^*,M,M^*}\circ (M^*\ot b_M)\nonumber\\
&=&
(d_M\ot M^*)\circ a^{-1}_{M^*,M,M^*}\circ (M^*\ot (f\ot g^*))\circ (M^*\ot b_M)\nonumber\\
&\equal{\equref{1.3.5}}&
(d_M\ot M^*)\circ a^{-1}_{M^*,M,M^*}\circ (M^*\ot ((f\circ g)\ot M^*))\circ (M^*\ot b_M)\nonumber\\
\eqlabel{1.3.6}
&=&
(d_M\ot M^*)\circ((M^*\ot (f\circ g))\ot M^*)\circ a^{-1}_{M^*,M,M^*}\circ (M^*\ot b_M).
\end{eqnarray}

\begin{proposition}\prlabel{1.3}
Suppose that $\Cc$ is a monoidal category with left duality. Then $\widetilde{\Hh}(\Cc)$ also has left
duality. The left dual of $(M,\mu)\in \widetilde{\Hh}(\Cc)$ is $(M^*,(\mu^*)^{-1})$, and the
left and right coevaluation maps are given by the formulas
$$\tilde{d}_M=d_M\circ (\mu^*\ot \mu)~~;~~\tilde{b}_M=(\mu\ot \mu^*)^{-1}\circ b_M.$$
\end{proposition}

\begin{proof}
Using \equref{1.3.5}, we obtain that
$$(\mu\ot (\mu^*)^{-1})\circ b_M= ((\mu\circ \mu^{-1})\ot M^*)\circ b_M=b_M,$$
so $b_M$ is a morphism in $\Hh(\Cc)$. Then $\tilde{b}_M$ is also a morphism
in $\Hh(\Cc)$.
In a similar way, $\tilde{d}_M$ is a morphism in $\Hh(\Cc)$. Now we compute
\begin{eqnarray*}
&&\hspace*{-1cm}
\tilde{r}_M\circ (M\ot \tilde{d}_M)\circ \tilde{a}_{M,M^*,M}\circ (\tilde{b}_M\ot M)\circ \tilde{l}_M^{-1}
\\
&=&
\mu\circ r_M\circ (M\ot d_M)\circ (M\ot (\mu^*\ot \mu)) \circ (\mu\ot (M^*\ot \mu^{-1}))\\
&&\circ {a}_{M,M^*,M}\circ ((\mu^{-1}\ot (\mu^*)^{-1})\ot M)
\circ (b_M\ot M)\circ l_M^{-1}\circ \mu^{-1}\\
&=&
\mu\circ r_M\circ (M\ot d_M)\circ {a}_{M,M^*,M} 
\circ (b_M\ot M)\circ l_M^{-1}\circ \mu^{-1}\\
&\equal{\equref{1.3.1}}& \mu\circ M\circ \mu^{-1}=M;
\end{eqnarray*}
\begin{eqnarray*}
&&\hspace*{-1cm}
\tilde{l}_{M^*}\circ (\tilde{d}_M\ot M^*)\circ \tilde{a}^{-1}_{M^*,M,M^*}\circ (M^*\ot \tilde{b}_M)\circ 
\tilde{r}_{M^*}^{-1}\\
&=&
(\mu^*)^{-1}\circ {l}_{M^*}\circ ({d}_M\ot M^*)\circ ((\mu^*\ot \mu)\ot M^*)
\circ ((\mu^*\ot M)\ot (\mu^*)^{-1})\\
&& \circ a^{-1}_{M^*,M,M^*}\circ (M^*\ot (\mu^{-1}\ot (\mu^*)^{-1}))
\circ (M^*\ot b_M)\circ r_{M^*}^{-1}\circ \mu^*\\
&=& 
(\mu^*)^{-1}\circ {l}_{M^*}\circ ({d}_M\ot M^*)\circ (((\mu^*)^2\ot M)\ot (\mu^*)^{-2})\\
&& \circ a^{-1}_{M^*,M,M^*}\circ (M^*\ot b_M)\circ r_{M^*}^{-1}\circ \mu^*\\
&\equal{\equref{1.3.6}}&
(\mu^*)^{-1}\circ {l}_{M^*}\circ  ({d}_M\ot M^*)\circ a^{-1}_{M^*,M,M^*}\circ (M^*\ot b_M)\circ r_{M^*}^{-1}\circ \mu^*\\
&\equal{\equref{1.3.2}}&
(\mu^*)^{-1}\circ M^*\circ \mu^*=M^*.
\end{eqnarray*}
\end{proof}

In the same way, we can prove the following result.

\begin{proposition}\prlabel{1.4}
Suppose that $\Cc$ is a monoidal category with right duality. Denote the
right dual of $M\in \Cc$ by ${}^*M$, and the evaluation and coevalution maps
by $d'_M$ and $b'_M$. Then $\Hh(\Cc)$ also has right duality. The right
 dual of $(M,\mu)\in \widetilde{\Hh}(\Cc)$ is $({}^*M,({}^*\mu)^{-1})$, and the
left and right coevaluation maps are given by the formulas
$$\tilde{d}'_M=d'_M\circ ({}^*\mu\ot \mu)^{-1}~~;~~\tilde{b}'_M=(\mu\ot {}^*\mu)\circ b'_M.$$
\end{proposition}

Obviously, the Hom-construction is functorial: for a functor $F:\ \Cc\to \Dd$, we have
a functor $\Hh(F):\ \Hh(\Cc)\to \Hh(\Dd)$ given by
$$\Hh(F)(M,\mu)=(F(M),F(\mu))~~;~~\Hh(F)(f)=F(f).$$

\begin{proposition}\prlabel{1.5}
Let $\Cc=(\Cc,\ot,I,a,l,r)$ and $\Dd=(\Dd,\ot,J,b,m,s)$ be monoidal categories, and
that $(F,\varphi_0,\varphi_2)$ is a monoidal functor $\Cc\to \Dd$.
Then $(\Hh(F),\varphi_0,\varphi_2)$ is a monoidal functor $\widetilde{\Hh}(\Cc)\to \widetilde{\Hh}(\Dd)$.
\end{proposition}

\begin{proof}
$\varphi_0:\ J\to F(I)$ is clearly a morphism in $\Hh(\Dd)$. For all $(M,\mu),(N,\nu)\in \Hh(\Cc)$,
$$\varphi_2(M,N):\ F(M)\ot F(N)\to F(M\ot N)$$
is a morphism in $\Hh(\Dd)$, since $\varphi_2$ is natural in $M$ and $N$, so that
$$\varphi_2(M,N)\circ (F(\mu)\ot F(\nu))=F(\mu\ot\nu)\circ \varphi_2(M,N).$$
Let us now show that the following diagram commutes
$$\xymatrix{
(F(M)\ot F(N))\ot F(P)\ar[rr]^{\tilde{b}_{F(M),F(N),F(P)}}\ar[d]_{\varphi_2(M,N)\ot F(P)}&&
F(M)\ot (F(N)\ot F(P))\ar[d]^{F(M)\ot \varphi_2(N,P)}\\
F(M\ot N)\ot F(P)\ar[d]_{\varphi_2(M\ot N,P)}&&
F(M)\ot F(N\ot P)\ar[d]^{\varphi_2(M,N\ot P)}\\
F((M\ot N)\ot P)\ar[rr]^{F(\tilde{a}_{M,N,P})}&&
F(M\ot (N\ot P))}$$
Indeed,
\begin{eqnarray*}
&&\hspace*{-1cm}
\varphi_2(M,N\ot P)\circ (F(M)\ot \varphi_2(N,P))\circ\tilde{b}_{F(M),F(N),F(P)}\\
&=&
\varphi_2(M,N\ot P)\circ (F(M)\ot \varphi_2(N,P))\\
&&\hspace*{1cm}\circ (F(\mu)\ot (F(N)\ot F(\pi^{-1})))
\circ{b}_{F(M),F(N),F(P)}\\
&=&
\varphi_2(M,N\ot P)\circ (F(\mu)\ot F(N\ot \pi^{-1}))\\
&&\hspace*{1cm} \circ (F(M)\ot \varphi_2(N,P))
\circ{b}_{F(M),F(N),F(P)}\\
&=& F(\mu\ot (N\ot \pi^{-1}))\circ \varphi_2(M,N\ot P)\\
&&\hspace*{1cm} \circ (F(M)\ot \varphi_2(N,P))
\circ{b}_{F(M),F(N),F(P)}\\
&=& F(\mu\ot (N\ot \pi^{-1}))\circ F(a_{M,N,P})\circ \varphi_2(M\ot N,P)\circ
\varphi_2(M,N)\circ F(P)\\
&=& F\bigl( (\mu\ot (N\ot \pi^{-1}))\circ a_{M,N,P}\bigr)
\circ \varphi_2(M\ot N,P)\circ
\varphi_2(M,N)\circ F(P)\\
&=& F(\tilde{a}_{M,N,P}) \circ \varphi_2(M\ot N,P)\circ
\varphi_2(M,N)\circ F(P).
\end{eqnarray*}
\end{proof}

\begin{proposition}\prlabel{1.6}
We have an isomorphism of categories
$$F:\ \Hh(\Cc^{\rm op})\to \Hh(\Cc)^{\rm op},~~F(M,\mu)=(M,\mu^{-1}),~F(f)=f.$$
If $\Cc$ is a monoidal, then $F$ defines monoidal isomorphisms
$\Hh(\Cc^{\rm op})\cong \Hh(\Cc)^{\rm op}$ and 
$\widetilde{\Hh}(\Cc^{\rm op})\cong \widetilde{\Hh}(\Cc)^{\rm op}$.
\end{proposition}

Let $\Cc$ be a monoidal category.
Our next aim is to show that the categories $\Hh(\Cc)=
(\Hh(\Cc),\ot, (I,I), a,l,r)$ and $\widetilde{\Hh}(\Cc)=
(\Hh(\Cc),\ot, (I,I), \tilde{a},\tilde{l},\tilde{r})$ are tensor isomorphic.
Let $F:\ \Hh(\Cc)\to \Hh(\Cc)$ be the identity functor, and
$\varphi_0:\ I\to I$ the identity. For $M,N\in \Hh(\Cc)$, we define
$$\varphi_2(M,N)=\mu\ot\nu:\ F(M)\ot F(N)=F(M\ot N)\to F(M\ot N)=M\ot N.$$
$\varphi_2$ is natural in $M$ and $N$: if $f:\ M\to M'$ and $g:\ N\to N'$
are morphisms in $\Hh(\Cc)$, then we have the commutative diagram
$$\xymatrix{
M\ot N\ar[rr]^{\mu\ot \nu}\ar[d]_{f\ot g}&&M\ot N\ar[d]^{f\ot g}\\
M'\ot N'\ar[rr]^{\mu'\ot \nu'}&& M'\ot N'}$$

\begin{proposition}\prlabel{1.7}
Let $\Cc$ be a monoidal category. Then the functor $(F,\varphi_0,\varphi_2):
\Hh(\Cc)\to \tilde{\Hh}(\Cc)$
defined above is strong monoidal. Consequently, the monoidal categories
$\Hh(\Cc)$ and $\tilde{\Hh}(\Cc)$ are tensor isomorphic.
\end{proposition}

\begin{proof}
We have to show first that the diagram
$$\xymatrix{
(M\ot N)\ot P\ar[rr]^{\tilde{a}_{M,N,P}}\ar[d]_{\varphi_2(M,N)\ot P}&&
M\ot (N\ot P)\ar[d]^{M\ot \varphi_2(N,P)}\\
(M\ot N)\ot P \ar[d]_{\varphi_2(M\ot N,P)}&&
M\ot (N\ot P)\ar[d]^{\varphi_2(M,N\ot P)}\\
(M\ot N)\ot P\ar[rr]^{{a}_{M,N,P}}&& M\ot (N\ot P)}$$
commutes, for all $(M,\mu),(N,\nu),(P,\pi)\in \Hh(\Cc)$. Indeed,
\begin{eqnarray*}
&&\hspace*{-2cm}
\varphi_2(M,N\ot P)\circ (M\ot \varphi_2(N,P))\circ \tilde{a}_{M,N,P}\\
&=&
(\mu\ot(\nu\ot \pi))\circ (M\ot (\nu\ot \pi))\circ (\mu\ot (N\ot \pi^{-1}))\circ {a}_{M,N,P}\\
&=& (\mu^2\ot (\nu^2\ot \pi))\circ {a}_{M,N,P}\\
&=& {a}_{M,N,P} \circ ((\mu^2\ot \nu^2)\ot \pi)\\
&=&  {a}_{M,N,P} \circ ((\mu\ot \nu)\ot \pi \circ ((\mu\ot \nu)\ot P)\\
&=&  {a}_{M,N,P} \circ \varphi_2(M\ot N,P)\circ (\varphi_2(M,N)\ot P).
\end{eqnarray*}
Finally we need to show that the following two diagrams commute:
$$\xymatrix{
I\ot M\ar[rr]^{\tilde{l}_M}\ar[d]_{\varphi_0\ot M}&&M\\
I\ot M\ar[rr]^{\varphi_2(I,M)}&&I\ot M\ar[u]_{l_M}}~~~~
\xymatrix{
M\ot I\ar[rr]^{\tilde{r}_M}\ar[d]_{M\ot \varphi_0}&&M\\
M\ot I\ar[rr]^{\varphi_2(M,I)}&&M\ot I\ar[u]_{r_M}}$$
Indeed,
$$l_M\circ \varphi_2(I,M)\circ \varphi_0\ot M=l_M\ot (I\ot \mu)=\tilde{l}_M.$$
The commutativity of the second diagram is proved in a similar way.
\end{proof}

\begin{remark}\relabel{Homadditivity}
Let $\Cc$ be an additive category. It is easy to see that both ${\Hh}(\Cc)$ and $\widetilde{\Hh}(\Cc)$ are additive as well.
\end{remark}

\prref{1.8} allows us to characterize algebras and coalgebras in the category $\widetilde{\Hh}(\Cc)$. These descriptions will be very useful in the following Section.
\\Since the functor $(F,\varphi_0,\varphi_2)$ from \prref{1.7} is a strong monoidal
isomorphism of monoidal categories, we obtain a category isomorphism between
the category of algebras in $\Hh(\Cc)$ and the category of algebras in $\widetilde{\Hh}(\Cc)$.
Algebras in $\Hh(\Cc)$ are easy to describe: they are of the form $(A,\alpha)$,
where $A$ is an algebra in $\Cc$ and $\alpha:\ A\to A$ is an algebra automorphism.

\begin{corollary}\colabel{1.9}
An algebra $\tilde{A}$ in $\widetilde{\Hh}(\Cc)$ is of the following type:
\\$\tilde{A}=(A,\alpha,m_A\circ (\alpha\ot \alpha)=\alpha\circ m_A,\eta_A)$,
where $A$ is an algebra in $\Cc$ with multiplication $m_A$ and unit $\eta_A$,
and $\alpha$ is an algebra automorphism of $A$.
\end{corollary}

The functor $(F,\varphi_0,\varphi_2)$ from \prref{1.7} is a strong monoidal one, hence
$(F,\varphi^{-1}_0,\varphi^{-1}_2)$ is strong comonoidal. Now we have
a category isomorphism between
the category of coalgebras in $\Hh(\Cc)$ and the category of coalgebras in $\widetilde{\Hh}(\Cc)$.
Coalgebras in $\Hh(\Cc)$ are also easy to describe, and we obtain the following structure
theorem for coalgebras in $\widetilde{\Hh}(\Cc)$.

\begin{corollary}\colabel{1.10}
A coalgebra $\tilde{C}$ in $\widetilde{\Hh}(\Cc)$ is of the form
$\tilde{C}=(C,\gamma,(\gamma^{-1}\ot \gamma^{-1})\circ\Delta_C=\Delta_C\circ\gamma^{-1},
\varepsilon_C)$, where $(C,\Delta_C,\varepsilon_C)$ is a coalgebra in $\Cc$ and
$\gamma:\ C\to C$ is a coalgebra automorphism.
\end{corollary}

Corollaries \ref{co:1.9} and  \ref{co:1.10} are related to \cite[Theorem 2.4]{Yau2} and \cite[Theorem 3.16]{MS3}.
We will come back to this in \seref{2}, see the comments following \coref{2.4.3}.
 
We now also take braidings into account.

\begin{proposition}\prlabel{1.13}
Let $\Cc$ be a braided monoidal category. Then the monoidal functor $(F,\varphi_0,\varphi_2)$
from \prref{1.7} is braided. Consequently $\Hh(\Cc)$ and $\widetilde{\Hh}(\Cc)$ are
isomorphic as braided monoidal categories.
\end{proposition}

Using \prref{1.8} and \prref{1.12},
we now immediately obtain the following structure theorems for bialgebras and Hopf
algebras in $\widetilde{\Hh}(\Cc)$.

\begin{proposition}\prlabel{1.14}
A bialgebra $\tilde{H}$ in $\widetilde{\Hh}(\Cc)$ is of the following type:
$\tilde{H}=(H,\alpha, m_H\circ (\alpha\ot\alpha),\eta_A, \Delta_H\circ\alpha^{-1},\varepsilon_H)$,
where $(H,m_H,\eta_H,\Delta_H,\varepsilon_H)$ is a bialgebra in $\Cc$, and
$\alpha\in \Aut_\Cc(H)$ is a bialgebra automorphism. $\tilde{H}$ is a Hopf algebra
in $\widetilde{\Hh}(\Cc)$ if $H$ is a Hopf algebra in $\Cc$.
\end{proposition}

Remark that we do not have to impose that the antipode $S$ of $H$ commutes with
$\alpha$ (equivalently, $\alpha$ is a Hopf algebra isomorphism): this follows automatically
from the fact that $\alpha$ is a bialgebra isomorphism.

\section{Monoidal Hom-Hopf algebras}\selabel{2}
Let $k$ be a commutative ring.
The results of \seref{1} can be applied to the category of $k$-modules
(vector spaces if $k$ is a field);
$\Cc=\Mm_k$. 
\\Let $(M,\mu),(N,\nu),(P,\pi)\in\tilde{\Hh}(\Mm_k)$. The associativity
and unit constraints are given by the formulas
$$\tilde{a}_{M,N,P}((m\ot n)\ot p)=\mu(m)\ot (n\ot \pi^{-1}(p));$$
$$\tilde{l}_M(x\ot m)=\tilde{r}_M(m\ot x)=x\mu(m).$$
An algebra in $\widetilde{\Hh}(\Mm_k)$ will be called a \textit{monoidal Hom-algebra}\index{algebra!monoidal Hom-}.

\begin{proposition}\prlabel{2.1}
A monoidal Hom-algebra is an object $(A,\alpha)\in \widetilde{\Hh}(\Mm_k)$ together with a $k$-linear map
$m_A:\ A\ot A\to A$ and an element $1_A\in A$ such that\index{unitality!Hom-}\index{associativity!Hom-}
\begin{equation}\eqlabel{2.1.1}
\alpha(ab)=\alpha(a)\alpha(b)~~;~~\alpha(1_A)=1_A,
\end{equation}
\begin{equation}\eqlabel{2.1.2}
\alpha(a)(bc)=(ab)\alpha(c)~~;~~a1_A=1_Aa=\alpha(a),
\end{equation}
for all $a,b,c\in A$. Here we use the notation $m_A(a\ot b)=
ab$.
\end{proposition}

\begin{proof}
Let $(A,\alpha,m_A,\eta)$ be an algebra in $\widetilde{\Hh}(\Mm_k)$. Let $\eta(1)=1_A$.
The fact that the structure maps $m_A$ and $\eta_A$ are maps in $\widetilde{\Hh}(\Mm_k)$
is expressed by \equref{2.1.1}
The associativity of the multiplication map $m$ is expressed by the commutativity of the
following diagram:
$$\xymatrix{
(A\ot A)\ot A \ar[rr]^(.6){m_A\ot A}\ar[dd]_{\tilde{a}_{A,A,A}}&&A\ot A\ar[rd]^{m_A}\\
&&&A\\
A\ot(A\ot A) \ar[rr]^(.6){A\ot m_A}&&A\ot A\ar[ru]^{m_A}}$$
or, for all $a,b,c\in A$:
$$(ab)c=\alpha(a)(b\alpha^{-1}(c)),$$
which is equivalent to the first formula of \equref{2.1.2}. 
The unit condition is expressed by the commutativity of the diagram
$$\xymatrix{
k\ot A\ar[rr]^{\eta\ot A}\ar[rrd]_{l_A}&&A\ot A\ar[d]^{m_A}&&
A\ot k\ar[ll]_{A\ot \eta}\ar[lld]^{r_A}\\
&&A&&}$$
which is equivalent to the second formula of \equref{2.1.2}.
\end{proof}

\begin{remark}\relabel{2.2}
The Hom-associativity condition in \equref{2.1.2} first appeared in 
\cite[Def. 1.1]{MS}. In \cite{MS}, a Hom-(associative) algebra\index{algebra!Hom-} is a triple 
$(A;\mu_A; \alpha)$, with a linear space $A$, 
a linear space homomorphism (linear map) 
$\alpha:A \rightarrow A$ and a bilinear multiplication $\mu_A : A \otimes A \rightarrow A$ 
satisfying the Hom-associativity condition in the first equation in \equref{2.1.2}.
The authors of \cite{MS} do not impose that $\alpha$ is injective, nor that $\mu_A$ 
satisfies the first equation of \equref{2.1.1}, nor that the algebra has a unit in any sense. 
In \cite{MS} and \cite{Fregier2}, the authors consider also the subclass of unital
Hom-associative algebras. For the unital Hom-associative algebras they require the existence 
of $1_A \in A$ such that
$a1_A = 1_A a = a$, for all $a \in A$, a condition that is clearly different from the
second equation in \equref{2.1.2}. It should be also mentioned here that the 
unitality condition in \equref{2.1.2} has been introduced and studied 
in detail in \cite{Fregier2} where it is called weak unitality (see \cite[Def. 2.1]{Fregier2}).\\
Now let $(A,*,\alpha)$ be a Hom-associative structure in the classical sense of
\cite{MS}: $\alpha:\ A\to A$ is $k$-linear, and $*$ is a multiplication on $A$ satisfying the
first equality in \equref{2.1.2}. Gohr \cite{Gohr} calls $A$ a twist if there is an associative
multiplication $\cdot$ on $A$ such that $a*b=\alpha(a\cdot b)$, for all $a,b\in A$.
Sufficient conditions for $A$ to be a twist are discussed in great detail in
\cite{Fregier,Fregier2,FGS,Gohr}. For example, the twist property can be characterized
completely in the situation where $\alpha$ is surjective, see \cite[Prop. 2]{Gohr}.
In our situation, the condition that $\alpha$ is bijective is essential, see the definition
of the monoidal categories $\Hh(\Cc)$ and $\tilde{\Hh}(\Cc)$ in \seref{1}. Then
\prref{2.1} tells us that $A$ is a twist, and this is essentially a consequence of the
monoidal equivalence discussed in \prref{1.7}.
\end{remark}

From \coref{1.9}, we immediately obtain the following result. 

\begin{proposition}\prlabel{2.3}
A monoidal Hom-algebra is of the form $\tilde{A}=(A,\alpha,m_A\circ (\alpha\ot \alpha)
=\alpha\circ m_A,\eta_A)$ where $(A,m_A,\eta_A)$ is an algebra in the
usual sense, and $\alpha\in \Aut(A)$ is an algebra automorphism.
\end{proposition}

A coalgebra in $\widetilde{\Hh}(\Mm_k)$ will be called a \textit{monoidal Hom-coalgebra}\index{coalgebra!monoidal Hom-}.

\begin{proposition}\prlabel{2.4}
A monoidal Hom-coalgebra is an object $(C,\gamma)\in \widetilde{\Hh}(\Mm_k)$ together
with $k$-linear maps $\Delta:\ C\to C\ot C$,
$\Delta(c)=c_{(1)}\ot c_{(2)}$ (summation implicitly understood) and $\varepsilon:\ C\to k$ such that
\begin{equation}\eqlabel{2.4.1}
\Delta(\gamma(c))=\gamma(c_{(1)})\ot \gamma(c_{(2)})~~;~~
\varepsilon(\gamma(c))=\varepsilon(c);
\end{equation}
and
\begin{equation}\eqlabel{2.4.2}
\gamma^{-1}(c_{(1)})\ot \Delta(c_{(2)})=\Delta(c_{(1)})\ot \gamma^{-1}(c_{(2)})~~;~~
c_{(1)}\varepsilon(c_{(2)})=\varepsilon(c_{(1)})c_{(2)}=\gamma^{-1}(c),
\end{equation}
for all $c\in C$.
\end{proposition}

\begin{proof}
Let $(C,\gamma,\Delta,\varepsilon)$ be a coalgebra in $\widetilde{\Hh}(\Mm_k)$.
\equref{2.4.1} expresses the fact that $\Delta$ and $\varepsilon$ are morphisms in
$\widetilde{\Hh}(\Mm_k)$. The coassociativity of $\Delta$ is equivalent to the 
commutativity of the following diagram
$$\xymatrix{
&C\ot C\ar[rr]^(.4){\Delta\ot C}&&(C\ot C)\ot C\ar[dd]^{\tilde{a}_{C,C,C}}\\
C\ar[ru]^{\Delta}\ar[rd]_{\Delta}&&&\\
&C\ot C\ar[rr]^(.4){C\ot \Delta}&&C\ot (C\ot C)}$$
or
$$\gamma(c_{(1)(1)})\ot c_{(1)(2)}\ot \gamma^{-1}(c_{(2)})=
c_{(1)}\ot \Delta(c_{(2)}),$$
for all $c\in C$, which is clearly equivalent to the first formula in \equref{2.4.2}.
The counit property is handled in a similar way.
\end{proof}

From \coref{1.10}, we immediately obtain:

\begin{corollary}\colabel{2.4.3}
A monoidal Hom-coalgebra is of the form
$\tilde{C}=(C,\gamma,(\gamma^{-1}\ot \gamma^{-1})\circ\Delta=\Delta\circ\gamma^{-1},
\varepsilon)$, where $(C,\Delta,\varepsilon)$ is a coalgebra and $\gamma:\ C\to C$
is a coalgebra automorphism.
\end{corollary}

Hom-coalgebras\index{coalgebra!Hom-} were introduced by Makhlouf and Silvestrov in \cite{MS2},
and have been studied in \cite{MS3} and \cite{Yau3}. In \cite[Theorem 3.16]{MS3}, it
is observed that $\tilde{C}$, as considered in \coref{2.4.3} is a Hom-coalgebra.

A \textit{monoidal Hom-bialgebra}\index{bialgebra!monoidal Hom-} $H=(H,\alpha,m,\eta,\Delta,\varepsilon)$ is a bialgebra in
the symmetric monoidal category $\widetilde{\Hh}(\Mm_k)$. This means that
$(H,\alpha,m,\eta)$ is a monoidal Hom-algebra, $(H,\alpha,\Delta,\varepsilon)$ is a monoidal Hom-coalgebra
and that $\Delta$ and $\varepsilon$
are morphisms of monoidal Hom-algebras, that is
$$\Delta(bb')=\Delta(b)\Delta(b')~~;~~\Delta(1_B)=1_B\ot 1_B;$$
$$\varepsilon(bb')=\varepsilon(b)\varepsilon(b')~~;~~\varepsilon(1_B)=1.$$
From \prref{1.14}, it follows that a monoidal Hom-bialgebra is of the form
$H=(H,\alpha,m\circ(\alpha\ot \alpha),\eta,\Delta\circ\alpha^{-1},\varepsilon)$,
where $(H,m,\eta,\Delta,\varepsilon)$ is a bialgebra and $\alpha:\ H\to H$
is a bialgebra automorphism.\\

We now present a categorical interpretation of the definition of a monoidal Hom-bialgebra.
We can consider modules over a Hom-algebra $A=(A,\alpha)$. A left $(A,\alpha)$-\textit{Hom-module}\index{module!Hom-}
consists of $(M,\mu)\in \widetilde{\Hh}(\Mm_k)$ together with a morphism
$\psi:\ A\ot M\to M$, $\psi(a\ot m)=am$, in $\widetilde{\Hh}(\Mm_k)$ such that
\begin{equation}\eqlabel{2.4.4.1}
\alpha(a)(bm)=(ab)\mu(m)~~{\rm and}~~1_Am=\mu(m),
\end{equation}
for all $a\in A$ and $m\in M$. The fact that $\psi\in \widetilde{\Hh}(\Mm_k)$ means
that 
\begin{equation}\eqlabel{2.4.4.2}
\mu(am)=\alpha(a)\mu(m).
\end{equation}
A morphism $f:\ (M,\mu)\to (N,\nu)$ in
$\widetilde{\Hh}(\Mm_k)$ is called left $A$-linear if it preserves the $A$-action,
that is, $f(am)=af(m)$. ${}_A\widetilde{\Hh}(\Mm_k)$ will denote the category of
left $(A,\alpha)$-Hom-modules and $A$-linear morphisms.

\begin{proposition}\prlabel{2.4.4}
Let $H$ be a monoidal Hom-bialgebra. Then ${}_H\widetilde{\Hh}(\Mm_k)$ has a monoidal structure
such that the forgetful functor
${}_H\widetilde{\Hh}(\Mm_k)\to \widetilde{\Hh}(\Mm_k)$ is strong monoidal.
\end{proposition}

\begin{proof}
Take $(M,\mu),(N,\nu),(P,\pi)\in {}_H\widetilde{\Hh}(\Mm_k)$. 
$(M\ot N,\mu\ot\nu)\in {}_H\widetilde{\Hh}(\Mm_k)$ via the left
$H$-action
$$h\cdot (m\ot n)=h_{(1)}m\ot h_{(2)}n.$$
Let us verify the associativity and unit conditions:
for all $h,k\in H$, $m\in M$ and $n\in N$, we have
\begin{eqnarray*}
&&\hspace*{-13mm}
\alpha(h)\cdot (k\cdot (m\ot n))=
\alpha(h)\cdot (k_{(1)}m\ot k_{(2)}n)\\
&=& \alpha(h)_{(1)}(k_{(1)}m)\ot \alpha(h)_{(2)}(k_{(2)}n)
\equal{\equref{2.4.1}}
\alpha(h_{(1)})(k_{(1)}m)\ot \alpha(h_{(2)})(k_{(2)}n)\\
&\equal{\equref{2.4.4.1}}&(h_{(1)}k_{(1)})\mu(m)\ot (h_{(2)}k_{(2)})\nu(n)
=(hk)\cdot ((\mu\ot \nu)(m\ot n));\\
&&\hspace*{-13mm}
1_H\cdot (m\ot n)=1_H\cdot m\ot 1_H\cdot n=(\mu\ot \nu)(m\ot n).
\end{eqnarray*}
\equref{2.4.4.2} is satisfied since
\begin{eqnarray*}
&&\hspace*{-13mm}
(\mu \ot \nu)(h\cdot (m\ot n))=\mu(h_{(1)}m)\ot \nu(h_{(2)}n)\\
&=& \alpha(h_{(1)})\mu(m)\ot \alpha(h_{(2)})\nu(n)=\alpha(h)\cdot ((\mu\ot \nu)(m\ot n)).
\end{eqnarray*}
On $k$, we define a left $H$-action as follows: $h\cdot x=\varepsilon(h)x$. It
is straightforward to show that $(k,k)\in {}_H\widetilde{\Hh}(\Mm_k)$. Let us also show that
the associativity and unit constraint are left $H$-linear. For all $h\in H$,
$m\in M$, $n\in N$ and $p\in P$, we have
\begin{eqnarray*}
&&\hspace*{-2cm}
a_{M,N,P}(h\cdot ((m\ot n)\ot p))=
a_{M,N,P}((h_{(1)(1)}m\ot h_{(1)(2)}n)\ot h_{(2)}p)\\
&=& \mu(h_{(1)(1)}m)\ot (h_{(1)(2)}n\ot \pi^{-1}(h_{(2)}p))\\
&=&\mu(\alpha^{-1}(h_{(1)})m)\ot (h_{(2)(1)}n)\ot \pi^{-1}(\alpha(h_{(2)(2)})p))\\
&=& h_{(1)}\mu(m)\ot (h_{(2)(1)}n\ot h_{(2)(2)}\pi^{-1}(p))\\
&=& h\cdot (\mu(m)\ot (n\ot \pi^{-1}(p)))\\
&=& h\cdot a_{M,N,P}((m\ot n)\ot p);\\
&&\hspace*{-2cm}
l_M(h\cdot (x\ot m))=\varepsilon(h)x\mu(m)= h\cdot l_M(x\ot m).
\end{eqnarray*}
\end{proof}

\begin{remark}\relabel{2.4.5}
For a bialgebra in $\Mm_k$, we also have the converse property. Let $H$ be a
$k$-algebra, and assume that we have a monoidal structure on ${}_H\Mm$, the category of left $H$-modules, such that
the forgetful functor ${}_H\Mm\to \Mm$ is strong monoidal. Then a bialgebra structure
on $H$ is defined as follows: $\Delta(h)=h\cdot (1_H\ot 1_H)$ and $\varepsilon(h)=
h\cdot 1_k$. This bialgebra structure determines the monoidal structure on ${}_H\Mm$
completely: for $m\in M$, $n\in N$, define $f_m:\ H\to M$, $g_N:\ H\to N$ by
$f_m(h)=h\cdot m$, $g_n(h)=h\cdot n$. From the fact that $f_m\ot g_n\in {}_H\Mm$,
it then follows that $h\cdot (m\ot n)=h_{(1)}\cdot m\ot h_{(2)}\cdot n$.\\
In the Hom case, the first part of this argument still applies: we have a 
monoidal Hom-bialgebra
structure on $H$. The problem is that it does not determine the monoidal structure
completely; $f_m$ and $g_n$ are not morphisms in $\widetilde{\Hh}(\Mm_k)$.
Remark also that $(H,\alpha)$ is not a generator of ${}_H\widetilde{\Hh}(\Mm_k)$.
\end{remark}

Now let $C=(C,\gamma)$ be a monoidal Hom-coalgebra. A right $(C,\gamma)$-\textit{Hom-comodule}\index{comodule!Hom-}
is an object $(M,\mu)\in \widetilde{\Hh}(\Mm_k)$ together with a $k$-linear map
$\rho:\ M\to M\ot C$, notation $\rho(m)=m_{[0]}\ot m_{[1]}$ in $\widetilde{\Hh}(\Mm_k)$
such that 
$$\mu^{-1}(m_{[0]})\ot \Delta_{C}(m_{[1]})=m_{[0][0]}\ot (m_{[0][1]} \ot \gamma^{-1}(m_{[1]}))~~{\rm and}~~
m_{[0]}\varepsilon(m_{[1]})=\mu^{-1}(m)$$
for all $m\in M$. The fact that $\rho\in \widetilde{\Hh}(\Mm_k)$ means that
$$\rho(\mu(m))=\mu(m_{[0]})\ot \gamma(m_{[1]}).$$
Morphisms of right $(C,\gamma)$-Hom-comodule are defined in the obvious way.
The category of right $(C,\gamma)$-Hom-comodules will be denoted by 
$\widetilde{\Hh}(\Mm_k)^C$.

\begin{proposition}\prlabel{2.4.5}
Let $H$ be a monoidal Hom-bialgebra. Then $\widetilde{\Hh}(\Mm_k)^H$ has a monoidal structure
such that the forgetful functor
$\widetilde{\Hh}(\Mm_k)^H\to \widetilde{\Hh}(\Mm_k)$ is strong monoidal.
\end{proposition}

\begin{proof}
Take $(M,\mu),(N,\nu)\in \widetilde{\Hh}(\Mm_k)^H$. Then $(M\ot N,\mu\ot\nu)$
is a Hom-comodule:
$$\rho(m\ot n)=m_{[0]}\ot n_{[0]}\ot m_{[1]}n_{[1]}.$$
\end{proof}

Let $(C,\gamma)$ be a Hom-coalgebra, and $(A,\alpha)$ a Hom-algebra.
Then $(\Hom(C,A),$ $\alpha\circ-\circ\gamma^{-1})$ is a Hom-algebra with the
convolution as product, and unit $\eta\circ\varepsilon$. This result can be found
in \cite[Theorem 3.15]{MS2} and \cite[Prop. 3.4]{MS3}. The proof is a straightforward computation. 

For $M,N\in \widetilde{\Hh}(\Mm_k)$, let $\Hom^\Hh(M,N)$ be the submodule of
$\Hom(M,N)$ consisting of $k$-linear maps $M\to N$ that are morphisms in
$\widetilde{\Hh}(\Mm_k)$.

Let $(C,\gamma)$ be a monoidal Hom-coalgebra, and $(A,\alpha)$ a  monoidal Hom-algebra.
It is then easy to show that
$\Hom^\Hh(C,A)$ is an associative algebra with unit $\eta\circ\varepsilon$.

A \textit{monoidal Hom-Hopf algebra}\index{Hopf algebra!monoidal Hom-} is by definition a Hopf algebra in $\widetilde{\Hh}(\Mm_k)$.
It follows that a Hom-bialgebra $(H,\alpha)$ is a monoidal Hom-Hopf algebra if there
exists a morphism $S:\ H\to H$ in $\widetilde{\Hh}(\Mm_k)$
(i.e. $S$ commutes with $\alpha$) such that
$S*H=H*S=\eta\circ \varepsilon$. Monoidal Hom-Hopf algebras are examples of
Hom-Hopf algebras\index{Hopf algebra!Hom-}, as introduced in \cite{MS2}.
We finish this Section with some elementary
properties of the antipode. The fact that $S(1)=1$ and $\varepsilon\circ S=\varepsilon$
has been observed before in \cite[Prop. 3.14]{MS2} and \cite[Prop. 3.23]{MS3}.

\begin{proposition}\prlabel{2.4.8}
Let $H$ be a monoidal Hom-Hopf algebra. Then
$$S(hk)=S(k)S(h)~~;~~S(1)=1;$$
$$\Delta(S(h))=S(h_{(2)})\ot S(h_{(1)})~~;~~\varepsilon\circ S=\varepsilon.$$
\end{proposition}

\begin{proof}
We will only prove the first statement.
Consider the morphisms $F,G\in \Hom^\Hh(H\ot H,H)$ given by
$$F(h\ot k)=S(k)S(h)~~;~~G(h\ot k)=S(hk).$$
We show that $G$ is a left inverse, and $F$ is a right inverse for the
multiplication map $m$. This implies that $F=G$, which is the first formula.
\begin{eqnarray*}
&&\hspace*{-2cm}
(G*m)(h\ot k)=S(h_{(1)}k_{(1)})h_{(2)}k_{(2)}\\
&=&S((hk)_{(1)})(hk)_{(2)}=
\varepsilon(hk)1_H;\\
&&\hspace*{-2cm}
(m*F)(h\ot k)= (h_{(1)}k_{(1)})(S(k_{(2)})S(h_{(2)}))\\
&=& ((\alpha^{-1}(h_{(1)})\alpha^{-1}(k_{(1)}))S(k_{(2)}))\alpha(S(h_{(2)}))\\
&=& (h_{(1)}(\alpha^{-1}(k_{(1)})\alpha^{-1}(S(k_{(2)})))\alpha(S(h_{(2)}))\\
&=& (h_{(1)} \alpha^{-1}(k_{(1)}S(k_{(2)})))\alpha(S(h_{(2)}))\\
&=& (h_{(1)} \alpha^{-1}(\varepsilon(k)1_H))\alpha(S(h_{(2)}))\\
&=& \varepsilon(k) (h_{(1)}1_H)\alpha(S(h_{(2)}))\\
&=& \varepsilon(k) \alpha(h_{(1)})\alpha(S(h_{(2)}))=\varepsilon(k)\varepsilon(h)1_H.
\end{eqnarray*}
\end{proof}

\section{Hom-Hopf modules and the Fundamental Theorem}\selabel{2a}
Throughout this Section $(H,\alpha)$ will be a monoidal Hom-Hopf algebra. A right $(H,\alpha)$-\textit{Hom-Hopf module}\index{module!Hom-Hopf} $(M,\mu)$ is defined as being a right $(H,\alpha)$-Hom-module which is a right $(H,\alpha)$-Hom-comodule as well, satisfying the following compatibility relation:
\begin{equation}\eqlabel{2.5.1}
\rho(mh)=m_{[0]}h_{(0)}\ot m_{[1]}h_{(1)},
\end{equation}
whenever $m\in M$ and $h\in H$. Considering the Hom-algebra morphism $\Delta:\ H\to H\ot H$, $(M\ot H, \mu \ot \alpha)$ becomes a right $(H,\alpha)$-Hom-module (by putting $(m\ot h)g=mg_{(1)}\ot hg_{(2)}$ for any $m\in M$ and $h,g\in H$). The compatibility relation \equref{2.5.1} then means that $\rho$ is a morphism of right $(H,\alpha)$-Hom-modules. \\
A morphism between two right $(H,\alpha)$-Hom-Hopf modules is a $k$-linear map which is a morphism in the categories $\widetilde{\Hh}(\Mm_k)_{H}$ and $\widetilde{\Hh}(\Mm_k)^{H}$ at the same time. $\widetilde{\Hh}(\Mm_k)_{H}^{H}$ will denote the category of  right $(H,\alpha)$-Hom-Hopf modules and morphisms between them.  \\

Let $(N,\nu)\in \widetilde{\Hh}(\Mm_k)$. Consider $(N\ot H, \nu \ot \alpha)$ and the following two $k$-linear maps:
$$\psi: (N\ot H)\ot H \to N\ot H,~~ \psi((n\ot h)\ot g)= \nu(n)\ot h\alpha^{-1}(g);$$
$$ \rho: N\ot H \to (N\ot H)\ot H,~~\rho(n\ot h)= (\nu^{-1}(n)\ot h_{(1)})\ot \alpha(h_{(2)}).$$
It is easily checked that $\psi$, resp. $\rho$ define a right $(H,\alpha)$-Hom-module, resp. right $(H,\alpha)$-Hom-comodule structure on $(N\ot H, \nu \ot \alpha)$, and that
 the compatibility relation \equref{2.5.1} is satisfied. This construction is functorial, so we have a functor:
$$F=(-\ot H, -\ot \alpha): \widetilde{\Hh}(\Mm_k) \to \widetilde{\Hh}(\Mm_k)_{H}^{H}.$$
Let $(M,\mu)$ be a right $(H,\alpha)$-Hom-Hopf module and consider the set
$$M^{{\rm co}H}=\{m\in M~|~\rho(m)={\mu}^{-1}(m)\ot 1\},$$ 
then $(M^{{\rm co}H},\mu_{|M^{{\rm co}H}})\in \widetilde{\Hh}(\Mm_k)$. This construction is also functorial, so we have a functor:
$$G= (-)^{{\rm co}H}: \widetilde{\Hh}(\Mm_k)_{H}^{H}  \to \widetilde{\Hh}(\Mm_k).$$

\begin{theorem}\thlabel{2a.1}
$(F,G)$ is a pair of inverse equivalences.
\end{theorem}

\begin{proof}
1) We first show that $(F,G)$ is a pair of adjoint functors.
The unit and the counit of the adjunction are defined as follows. For $(M,\mu)\in \widetilde{\Hh}(\Mm_k)_{H}^{H}$,
$\varepsilon_{(M,\mu)}: (M^{{\rm co}H}\ot H,\mu_{|M^{{\rm co}H}}\ot \alpha)\to (M,\mu)$
is defined by the formula
$$\varepsilon_{(M,\mu)}(m\ot h)=mh.$$
$\varepsilon_{(M,\mu)}$ is a morphism in $\widetilde{\Hh}(\Mm_k)_{H}$,
and it follows from the compatibility relation \equref{2.5.1} that it is a morphism in $\widetilde{\Hh}(\Mm_k)^{H}$ as well. For $(N,\nu)\in \widetilde{\Hh}(\Mm_k)$, we define
$\eta_{(N,\nu)}: (N,\nu)\to ((N\ot H)^{{\rm co}H}, (\nu\ot \alpha)_{|(N\ot H)^{{\rm co}H}})$
as follows:
$$\eta_{(N,\nu)}(n)={\nu}^{-1}(n)\ot 1.$$
$\eta_{(N,\nu)}$ is well-defined since $\rho_{N\ot H}({\nu}^{-1}(n)\ot 1)=({\nu}^{-2}(n)\ot 1)\ot \alpha(1)=({\nu}\ot {\alpha})^{-1}({\nu}^{-1}(n)\ot 1)\ot 1$, for all $n\in N$.\\
For $(N,\nu)\in \widetilde{\Hh}(\Mm_k)$ and $(M,\mu)\in \widetilde{\Hh}(\Mm_k)^{H}_{H}$,
$n\in N$, $h\in H$ and $m\in M$, we easily compute that
$$(\varepsilon_{F(N,\nu)}\circ F\eta_{(N,\nu)})(n\ot h)=({\nu}^{-1}(n)\ot 1)h=n\ot h;$$
$$(G\varepsilon_{(M,\mu)}\circ \eta_{G(M,\mu)})(m)=\mu^{-1}(m)1_{H}\equal{\equref{2.4.4.1}} \mu(\mu^{-1}(m))=m.$$
2) Let us now prove that $\varepsilon$ is a natural isomorphism.
Let $(M,\mu)$ be a right $(H,\alpha)$-Hom-Hopf module, and take $m\in M$.
Then $m_{[0]}S(m_{[1]})\in M^{{\rm co}H}$, since 
\begin{eqnarray*}
  &&\hspace*{-2cm}
  \rho(m_{[0]}S(m_{[1]}))
  =
  m_{[0][0]}S(m_{[1]})_{(1)}\ot m_{[0][1]}S(m_{[1]})_{(2)}\\
  &=&
  m_{[0][0]}S(m_{[1](2)})\ot m_{[0][1]}S(m_{[1](1)})\\
  &=&
  m_{[0][0]}\alpha^{-1}(S(m_{[1]}))\ot \alpha(m_{[0][1](1)}S(m_{[0][1](2)}))\\
  &=&
  m_{[0][0]}\alpha^{-1}(S(m_{[1]}))\ot \alpha(\eta(\varepsilon(m_{[0][1]})))\\
  &=&
  m_{[0][0]}\alpha^{-1}(S(m_{[1]}))\ot \varepsilon(m_{[0][1]})1_{H}\\
  &=&
  \mu^{-1}(m_{[0]})\alpha^{-1}(S(m_{[1]}))\ot \alpha^{-1}(1_{H})\\
  &=&
  \mu^{-1}(m_{[0]}S(m_{[1]}))\ot 1_{H}.
\end{eqnarray*}
Now we define $\alpha: (M,\mu)\to (M^{{\rm co}H}\ot H,\mu_{|M^{{\rm co}H}}\ot \alpha)$ as follows:
$$\alpha(m)=m_{[0][0]}S(m_{[0][1]})\ot m_{[1]}.$$
$\alpha$ is the inverse of $\varepsilon_{(M,\mu)}$: for all $m\in M$, we have
\begin{eqnarray*}
  &&\hspace*{-1cm}
 \varepsilon_{(M,\mu)}(\alpha(m))
  =
 \varepsilon_{(M,\mu)}(m_{[0][0]}S(m_{[0][1]})\ot m_{[1]})
  =
 \varepsilon_{(M,\mu)}(m_{[0]}\ot \eta(\varepsilon(m_{[1]})))\\
  &=&
    \varepsilon_{(M,\mu)}(m_{[0]}\varepsilon(m_{[1]})\ot1_{H})
  =
  \varepsilon_{(M,\mu)}({\mu}^{-1}(m)\ot1_{H})=m.
\end{eqnarray*}
and, for $m'\in M^{{\rm co}H}$ and $h\in H$,
\begin{eqnarray*}
  &&\hspace*{-1cm}
  \alpha(\varepsilon_{(M,\mu)}(m'\ot h))
 =
  (m'h)_{[0][0]}S((m'h)_{[0][1]})\ot (m'h)_{[1]}\\
  &=&
  (\mu^{-1}(m')_{[0]}h_{(1)(1)})S(\mu^{-1}(m')_{[1]}h_{(1)(2)})\ot 1_{H}h_{(2)}\\
  &=&
  (\mu^{-1}(m'_{[0]})h_{(1)(1)})S(\alpha^{-1}(m'_{[1]})h_{(1)(2)})\ot \alpha(h_{(2)})\\
  &=&
  (\mu^{-2}(m')h_{(1)(1)})S(\alpha(h_{(1)(2)}))\ot \alpha(h_{(2)})\\
  &=&
  \mu^{-1}(m')(h_{(1)(1)}S(h_{(1)(2)}))\ot \alpha(h_{(2)})\\
  &=&
  \mu^{-1}(m')1_{H}\ot \varepsilon(h_{(1)})\alpha(h_{(2)})
  =
  m'\ot h.
\end{eqnarray*}
3) Finally, we will show that $\eta$ is a natural isomorphism.
Take $(N,\nu)\in \widetilde{\Hh}(\Mm_k)$ and $n\ot h \in (N\ot H)^{{\rm co}H}$. Then
$$\rho(n\ot h)=(\nu^{-1}(n)\ot h_{(1)})\ot \alpha(h_{(2)})=(\nu^{-1}(n)\ot \alpha^{-1}(h))\ot 1_{H}.$$
Applying 
$(\nu\ot \varepsilon) \ot \alpha$ to both sides of this equation, we obtain that
$$(n\ot \varepsilon(h_{(1)}))\ot \alpha(h_{(2)})=(n\ot \varepsilon(h))\ot 1_{H}.$$
Now define 
$\beta:((N\ot H)^{{\rm co}H},{\nu\ot \alpha}_{|(N\ot H)^{{\rm co}H}})\to (N,\nu)$
by the formula
$$\beta(n\ot h)=\nu(n)\varepsilon(h).$$
$\beta$ is the inverse of $\eta_{(N,\nu)}$: for all $n\ot h \in (N\ot H)^{{\rm co}H}$ we have
\begin{eqnarray*}
  &&\hspace*{-1cm}
  \eta_{(N,\nu)}(\beta(n\ot h))
 =
  \nu^{-1}(\nu(n)\varepsilon(h))\ot 1_{H}
 =
  n\varepsilon(h)\ot 1_{H}\\
  &=&
  n\ot \varepsilon(h_{(1)}) \alpha(h_{(2)})
  =
  n\ot \alpha^{-1}(\alpha(h))
  =
  n\ot h
\end{eqnarray*}
and for all $n\in N$, we have
$(\beta \circ \eta_{(N,\nu)})(n)=\nu(\nu^{-1}(n)\varepsilon(1_{H}))=n$.
\end{proof}

\section{Hom-group algebras}\selabel{3}
The first example of a classical Hopf algebra is a group algebra. We will now generalize
this example to the Hom situation.\\
Recall that the category of sets $({\sf Sets},\times,\{*\},\tau)$ (where $\tau$ is the twist) is a symmetric (strict) monoidal category. The linearizing functor
$L:\ {\sf Sets}\to \Mm_k$, $L(X)=kX$ being the free $k$-module with basis $X$,
is strong monoidal, and preserves the symmetry.\\ Recall
that algebras in ${\sf Sets}$ are just monoids and that every set $X$ has a unique structure of coalgebra
in ${\sf Sets}$, namely $\delta(x)=(x,x)$ and $\varepsilon(x)=*$, for all $x\in X$.
This coalgebra structure makes every monoid a bialgebra in ${\sf Sets}$. Finally,
a Hopf algebra in ${\sf Sets}$ is a group. Applying the functor $L$, we obtain
algebras, coalgebras etc. in $\Mm_k$, as we recalled in \seref{Hopfalgbraidedmonoidal}.\\
Recall that the Hom-construction discussed in \seref{1} is functorial in the following sense (cf. \prref{1.4} and \prref{1.5}):
for a functor $F:\ \Cc\to \Dd$, we have a functor $\Hh(F):\ \Hh(\Cc)\to \Hh(\Dd)$
given by
$$\Hh(F)(M,\mu)=(F(M),F(\mu)),~~\Hh(F)(f)=F(f).$$
If $\Cc$ and $\Dd$ are monoidal categories, and $F$ is strong monoidal, then we
have strong monoidal functors
$$\Hh(F):\ \Hh(\Cc)\to \Hh(\Dd)~~{\rm and}~~\widetilde{\Hh}(F):\ 
\widetilde{\Hh}(\Cc)\to \widetilde{\Hh}(\Dd).$$
In particular, we have a strong monoidal functor
$\widetilde{\Hh}(L):\ \widetilde{\Hh}({\rm Sets})\to \widetilde{\Hh}(\Mm_k)$.\\
The results of \seref{1} give us structure Theorems for algebras, coalgebras,...
in $\widetilde{\Hh}({\sf Sets})$. For example, let $X$ be a set, and $\xi$ a
permutation of $X$. Consider the maps $\delta:\ X\to X\times X$,
$\varepsilon:\ X\to \{*\}$ given by
$\delta(x)=(\xi^{-1}(x),\xi^{-1}(x))$, $\varepsilon(x)=*$. Then $(X,\xi,\delta,\varepsilon)$
is a Hom-comonoid (that is, a coalgebra in $\widetilde{\Hh}({\sf Sets})$), and
every Hom-comonoid is of this type, by \coref{1.10}.\\
In a similar way, if $\varphi$ is an automorphism of a group $G$, then $(G,\varphi)$
with structure maps
$$g\cdot h=\varphi(gh),~~\eta(*)=1_G,~~\delta(g)=(\varphi^{-1}(g),\varphi^{-1}(g)),$$
$$\varepsilon(g)=*,~~S(g)=g^{-1},$$
is a Hom-group, that is a Hopf algebra in $\widetilde{\Hh}({\sf Sets})$.
Applying the linearizing functor $\widetilde{\Hh}(L)$ to a Hom-comonoid, a Hom-monoid, 
or a Hom-group, we obtain resp. a Hom-coalgebra, a Hom-bialgebra,
a Hom-Hopf algebra. The image under $\widetilde{\Hh}(L)$ of a Hom-group is called
a \textit{Hom-group algebra}\index{algebra!Hom-group}. It is the free $k$-module with basis $G$, and the above 
structure maps extended linearly.
\begin{remark}\relabel{homgroup}
Let $G$ be a group and let $\varphi$ and $\varphi'$ be two automorphisms of $G$ that belong to the same conjugacy class. 
It immediately follows from the above construction that $\widetilde{\Hh}(L)(G,\varphi)$ and $\widetilde{\Hh}(L)(G,\varphi')$ are isomorphic as Hom-algebras.
\end{remark}
\section{Monoidal Hom-Lie algebras}\selabel{4}\index{Lie algebra!monoidal Hom-}
As $\Mm_{k}$ is an additive, symmetric monoidal category and combining \prref{1.2} and \reref{Homadditivity}, we immediately have that $\widetilde{\Hh}(\Mm_k)$ is an additive, symmetric monoidal category as well. So we can compute Lie algebras in $\widetilde{\Hh}(\Mm_k)$. Let us see what comes out if we do so.
\\Let $(L,\alpha)\in \widetilde{\Hh}(\Mm_k)$, and $[-,-]:\ L\ot L\to L$ a morphism
in $\widetilde{\Hh}(\Mm_k)$ (that is, $[\alpha(x),\alpha(y)]=\alpha[x,y]$). Then
condition \equref{ASsymm} is equivalent to
$$[x,y]+[y,x]=0,$$
for all $x,y\in L$. Using same notation as in \seref{Hopfalgbraidedmonoidal}, we easily compute that
$$t_{c}(x\ot (y\ot z))=\alpha(z)\ot (\alpha^{-1}(x)\ot y);$$
$$w_{c}(x\ot (y\ot z))=\alpha(y)\ot (z\ot \alpha^{-1}(x)).$$
\equref{Jacsymm} is therefore equivalent to
$$[x,[y,z]]+[\alpha(y),[z,\alpha^{-1}(x)]]+ [\alpha(z),[\alpha^{-1}(x),y]]=0$$
or, replacing $x$ by $\alpha(x)$,
$$[\alpha(x),[y,z]]+[\alpha(y),[z,x]]+ [\alpha(z),[x,y]]=0,$$
which is the Hom-Jacobi identity.\index{identity!Hom-Jacobi}\index{Lie algebra!Hom-}
\begin{remark}\relabel{analogueLie}
One could ask whether an analogeous property of \prref{1.12} exists for Lie algebras in an additive, symmetric monoidal category (and thus, as a particular application, obtain similar results as in \seref{1} for monoidal Hom-Lie algebras). This question is addressed in the slightly more general setting of YB-Lie algebras in additive (not necessarily symmetric) monoidal categories in \cite{IV}.
\end{remark}
\section{The tensor Hom-algebra}\selabel{5}\index{algebra!tensor Hom-}
Let $\Cc$ be an abelian, cocomplete symmetric monoidal category such that, for any object $X$ in $\Cc$, the endofunctors $-\ot X$ and $X\ot -$ preserve denumerable coproducts.
We use the notation introduced in \seref{Hopfalgbraidedmonoidal}. Let $M\in \Cc$. We write $T^0(M)=I$, the unit object of $\Cc$, and $T^n(M)=
\ot^{t^n} (M,M,\cdots,M)$, and consider the coproduct
$$T(M)=\coprod_{n=0}^{\infty} T^n(M).$$
For $n,m\geq 0$, consider the morphism
\begin{eqnarray*}
&&\hspace*{-2cm}
b(t^n\vee t^m,t^{n+m})_{M,\cdots,M}:\
T^n(M)\ot T^m(M)=\ot^{t^n\vee t^m}(M,\cdots, M)\\
&\to&
T^{n+m}(M)= \ot^{t^{n+m}}(M,\cdots, M)\to T(M);
\end{eqnarray*}
in the case where $n=0$ or $m=0$, we consider
$$l_{T^m(M)}:\ I\ot T^m(M)\to T^m(M)~~;~~
r_{T^n(M)}:\ T^n(M)\ot I \to T^n(M).$$
Using the universal property of the coproduct, we obtain a morphism
$$m:\ T(M)\ot T(M)= \coprod_{n,m\geq 0} T^n(M)\ot T^m(M)\to T(M)$$
We also have the coproduct morphism $\eta:\ I\to T(M)$. Then $(T(M),m,\eta)$
is an algebra in $\Cc$. The associativity of $M$ follows from the uniqueness of
the natural transformation $b(\psi,\psi')$.\\
$T(M)$ satisfies the universal property of the tensor algebra. Let $A$ be an algebra
in $\Cc$, and $f:\ M\to A$ in $\Cc$. For every $n\geq 0$, we have the map
$$\ot^{t_n}(f,f,\cdots, f):\ \ot^{t^n}(M,\cdots, M)=T^n(M)\to \ot^{t^n}(A,\cdots, A).$$
We also have the multiplication map
$$m_A^n:\ \ot^{t^n}(A,\cdots, A)\to A.$$
Consider the composition
$$f_n=m_A^n\circ \ot^{t_n}(f,f,\cdots, f):\ T^n(M)\to A.$$
Using the universal property of the coproduct, we obtain
$\ol{f}:\ T(M)\to A$. $\ol{f}$ is an algebra map, and the diagram
$$\xymatrix{
M\ar[rr]^{i}\ar[rrd]_{f}&& T(M)\ar[d]^{\ol{f}}\\
&& A}$$
commutes. We can apply the universal property of the tensor algebra
to define a Hopf algebra structure on
$T(M)$, as in the classical case. Denote the tensor product of $T(M)$ with itself
by $T(M)\ol{\ot}T(M)$.
The morphism
$$r^{-1}_{T(M)}\circ i+l^{-1}_{T(M)}\circ i:\ M\to T(M)\ol{\ot} T(M)$$
in $\Cc$ induces $\Delta:\ T(M)\to T(M)\ol{\ot} T(M)$. The zero morphism
$0_{M,I}:\ M\to I$ induces $\varepsilon:\ T(M)\to I$. $\Hom_\Cc(M,M)$ is an abelian group,
and the identity morphism $M$ of $M$ has an oppositie $-M$. The map
$i\circ (-M):\ M\to T(M)^{\rm op}$ induces $S:\ T(M)\to T(M)^{\rm op}$.
\\It is easy to see that $\widetilde{\Hh}(\Mm_k)$ is an abelian category satisfying
the necessary conditions, so we can apply our construction to $(M,\mu)\in \widetilde{\Hh}(\Mm_k)$.
Now $T(M,\mu)=(T(M)=\coprod_{n=0}^{\infty} T^n(M), T(\mu))$, where $T(\mu)$ is
constructed as follows. For every $n\geq 0$, we have the automorphism
$\ot^{t_n}(\mu,\cdots,\mu)=T^n(\mu)$ of $T^n(M)$; also let $T^0(\mu)$ be the
identity map on $k$. For every $n\in \NN$,  consider
$i_n\circ T^n(\mu)$, where $i_n:\ T^n(M)\to T(M)$ is the natural inclusion. Applying
the universal property of the coproduct, we optain $T(\mu)$.\\
In order to describe the structure maps on $T(M)$, it is convenient to
identify $\ot^u(M,\cdots, M)$ and $\ot^{u'}(M,\cdots,M)$ using $b(u,u')$, for
any $u,u'\in U_n$. For example, take $m,n,p\in M$, and consider
$m\ot n\in T^2(M)$ and $p\in T^1(M)$. Then
$$(m\ot n)p=(m\ot n)\ot p=\mu(m)\ot (n\ot \mu^{-1}(p))\in T^3(M).$$
The unit element is $1\in k=T^0(M)$. On $M$, the cotensor product is defined as
follows:
$$\Delta(m)=1\ol{\ot} \mu^{-1}(m)+ \mu^{-1}(m)\ol{\ot} 1\in T(M)\ol{\ot} T(M).$$
The Hom-coassociativity of $\Delta$ can be verified directly for $m\in M$:
\begin{eqnarray*}&&\hspace*{-2cm}
(T(\mu)^{-1}\ol{\ot}\Delta)\Delta(m)\\
&=& 1\ol{\ot}1\ol{\ot}\mu^{-2}(m)+1\ol{\ot}\mu^{-2}(m) \ol{\ot} 1+ 
\mu^{-2}(m)\ol{\ot}1\ol{\ot}1\\
&=&
(\Delta\ol{\ot}T(\mu)^{-1})\Delta(m).
\end{eqnarray*}
$\Delta$ extends multiplicatively to $T(M)$. Let us finally compute $S$ on
$T^i(M)$ for $i=1,2,3$. Let $m,n,p\in M$.
$$S(m)=-m~~;~~S(m\ot n)=S(n)\ot S(m)=n\ot m;$$
$$S(m\ot (n\ot p))=S(n\ot p)\ot S(m)=-(p\ot n)\ot m=-\mu(p)\ot (n\ot \mu^{-1}(m)).$$

\section{Hom-ideals, coideals and Hopf ideals}\selabel{6}
Let $(M,\mu)\in \widetilde{\Hh}(\Mm_k)$. A $k$-submodule $N\subset M$ is
called a subobject of $(M,\mu)$ if $\mu$ restricts to an automorphism of $N$,
that is, $(N,\mu_{|N})\in \widetilde{\Hh}(\Mm_k)$. In this case,
$\mu$ induces an automorphism $\ol{\mu}$ of $M/N$ and
$(M/N,\ol{\mu})\in \widetilde{\Hh}(\Mm_k)$.\\
Now let $(A,\alpha)$ be a monoidal Hom-algebra. A subobject $I$ of $(A,\alpha)$ is
called a (two-sided) Hom-ideal of $(A,\alpha)$ if $(AI)A=A(IA)\subset I$. If
$X\subset A$ is a subset of $A$, then
$$I=\{\sum_i (a_i\alpha^{n_i}(x_i))b_i~|~
a_i,b_i\in A,~n_i\in \ZZ,~x_i\in X\}$$
is the Hom-ideal generated by $X$. Obviously, if $\alpha(X)=X$, then
$$I=\{\sum_i (a_ix_i)b_i~|~a_i,b_i\in A,~x_i\in X\}.$$
If $I$ is a Hom-ideal of $(A,\alpha)$, then $(A/I,\ol{\alpha})$ is a Hom-algebra.\\
Now let $(C,\gamma)$ be a monoidal Hom-coalgebra. A subobject $I$ of $(C,\gamma)$
is called a Hom-coideal if $\Delta(I)\subset I\ot C+C\ot I$ and $\varepsilon(I)=0$.
Then $(C/I,\ol{\gamma})$ is a Hom-coalgebra.\\
Let $(H,\alpha)$ be a monoidal Hom-Hopf algebra. A subobject $I$ of $(H,\alpha)$ is called
a Hom-Hopf ideal of $H$ if it is a Hom-ideal and a Hom-coideal and $S(I)\subset I$.
Then $(H/I,\ol{\alpha})$ is a Hom-Hopf algebra.

\section{The enveloping Hom-algebra of a monoidal Hom-Lie algebra}\selabel{7}
In \cite{Yau1}, the enveloping algebra of a Hom-Lie algebra is discussed.
One would expect that this is a Hom-Hopf algebra, but such a result does not
appear in \cite{Yau1}. Now consider a monoidal Hom-Lie algebra, this is
a Lie algebra $((L,\alpha),[,])$ in $\widetilde{\Hh}(\Mm_k)$, and take the tensor Hom-algebra
$(T(L),T(\alpha))$. Then
$$X=\{[x,y]-x\ot y-y\ot x~|~x,y\in L\}\subset T(L)$$
satisfies the condition $T(\alpha)(X)=X$, and we can consider the Hom-ideal $I$ 
generated by $X$. We now verify that $I$ is a Hom-Hopf ideal. Clearly
$\varepsilon(X)=0$. For $x,y\in L$, we easily compute that
\begin{eqnarray*}
&&\hspace*{-2cm}
\Delta([x,y]-x\ot y-y\ot x)
= ([x,y]-x\ot y +y\ot x)\ol{\ot} 1\\
&+& 1\ol{\ot} ([x,y]-x\ot y +y\ot x)\in I\ol{\ot} T(L)+T(L)\ol{\ot} I;\\
&&\hspace*{-2cm}
S([x,y]-x\ot y-y\ot x)=-([x,y]+y\ot x-x\ot y)\in I.
\end{eqnarray*}
Now we define the \textit{enveloping Hom-algebra}\index{algebra!enveloping Hom-} of $L$ as $(T(L)/I,\overline{T(\alpha)})$, which is a Hom-Hopf algebra.


\begin{thebibliography}{99}
\bibitem{Fregier}
Y. Fr\'egier, A. Gohr, On Hom type algebras, {\sl J. Gen. Lie 
Theory Appl.} {\bf 4} (2010), 16 p.

\bibitem{Fregier2}
Y. Fr\'egier, A. Gohr,  On unitality conditions for hom-associative algebras, preprint
arXiv:0904.4874.

\bibitem{FGS}
Y.  Fr\'egier, A.  Gohr, S. D. Silvestrov,  Unital algebras of 
Hom-associative type and surjective or injective twistings, {\sl J. Gen. Lie 
Theory Appl.} {\bf 3} (2009), 285--295.

\bibitem{Gohr}
A. Gohr, On Hom-algebras with surjective twisting, {\sl J. Algebra} {\bf 324} (2010),
1483--1491.

\bibitem{IV}
I. Goyvaerts, J. Vercruysse, A note on the categorification of Lie algebras,
{\sl Lie Theory and its applications in physics}, {\sl Springer proceedings in Mathematics and Statistics}
{\bf 36}, (2013), 541--550.

\bibitem{HS}
J. T. Hartwig, D. Larsson, S. D. Silvestrov, Deformations of Lie algebras using $\sigma$-derivations,
{\sl J. Algebra} {\bf 295} (2006), 314--361.

\bibitem{K}
C. Kassel, ``Quantum Groups", {\sl Graduate Texts Math.}
{\bf 155}, Springer Verlag, Berlin, 1995.

\bibitem{Khar}
V.K. Kharchenko, I.P. Shestakov, Generalizations of Lie algebras, {\sl Adv. Appl. Clifford Algebr.} {\bf 22} (2012), 721--743.

\bibitem{ML}
S. Mac Lane, ``Categories for the working mathematician",
second edition, {\sl Graduate Texts Math.}
{\bf 5}, Springer Verlag, Berlin, 1998.

\bibitem{MS}
A. Makhlouf, S. D. Silvestrov, Hom-algebra stuctures, {\sl J. Gen. Lie Theory Appl.}
{\bf 2} (2008), 51--64.

\bibitem{MS2}
A. Makhlouf, S. D. Silvestrov, Hom-Lie admissible Hom-coalgebras and Hom-Hopf
algebras, in ``Generalized Lie Theory in Mathematics, Physics and Beyond", S. Silvestrov, E. Paal, V. Abramov, A. Stolin (eds.). Springer, Berlin, 2009, Chapter 17, pp 189--206.

\bibitem{MS3}
A. Makhlouf, S. D. Silvestrov, Hom-algebras and hom-coalgebras, {\sl J. Algebra Appl.}
{\bf 9} (2010), 553--589.

\bibitem{Silvestrov}
S. D. Silvestrov,
Paradigm of quasi-Lie and quasi-Hom-Lie algebras and quasi-deformations,
in ``New techniques in Hopf algebras and graded ring theory", S. Caenepeel
and F. Van Oystaeyen (eds.),
K. Vlaam. Acad. Belgie Wet. Kunsten (KVAB), Brussels, 2007, 165--177.

\bibitem{Yau1} 
D. Yau, Enveloping algebras of Hom-Lie algebras, {\sl J. Gen. Lie Theory Appl.}
{\bf 2} (2008), 95--108.

\bibitem{Yau2}
D. Yau, Hom-algebras and homology, {\sl J. Lie Theory} {\bf 10} (2009), 409--421.

\bibitem{Yau3}
D. Yau, Hom-bialgebras and comodule Hom-algebras, {\sl Int. Electron. J. Algebra}
{\bf 8} (2010), 45--64.
\end{thebibliography}
\end{document}